\documentclass[a4, 11pt]{amsart}
\usepackage[active]{srcltx}
\usepackage{amsmath,amssymb,amscd,amsthm,graphicx,enumerate}

\usepackage{hyperref}
\hypersetup{breaklinks=true}
\usepackage[utf8]{inputenc}
\usepackage[T1]{fontenc}
\usepackage[all]{xy}
\usepackage{subfig}
\usepackage{xcolor}
\definecolor{ultrablue}{rgb}{0.0,0.0, 1}
\usepackage{hyperref}
\hypersetup{
linktoc=page,
colorlinks,
linkcolor={ultrablue},
citecolor={ultrablue},
urlcolor={ultrablue}
}
\usepackage{mathrsfs}
\usepackage{amsmath,amssymb,amsthm,a4wide}
\usepackage{enumerate}
\usepackage{color}
\usepackage{soul}
\usepackage{graphicx}

\makeatletter 
\def\l@subsection{\@tocline{2}{0pt}{3pc}{6pc}{}}
 \makeatother

\usepackage{tikz}
  \usetikzlibrary{calc,spy}
\usetikzlibrary{shapes.geometric, arrows}
\usepackage{pgfplots}
\usetikzlibrary{intersections, pgfplots.fillbetween}

\theoremstyle{plain}

\newtheorem*{theorem*}{Theorem}

\newtheorem{lemma}{Lemma}[section]
\newtheorem{theorem}[lemma]{Theorem}
\newtheorem{proposition}[lemma]{Proposition}

\theoremstyle{definition}

\newtheorem{definition}[lemma]{Definition}
\newtheorem{remark}[lemma]{Remark}

\newtheorem{innercustomthm}{Theorem}
\newenvironment{customthm}[1]
{\renewcommand\theinnercustomthm{#1}\innercustomthm}
{\endinnercustomthm}

\theoremstyle{remark}

\newtheorem*{ack}{Acknowledgements}

\numberwithin{equation}{section}

\newcommand{\ds}{\displaystyle}

\newcommand{\del}{{\partial}}

\newcommand{\sphere}{\mathrm{\mathbb{S}}}

\newcommand{\R}{\mathbb{R}}

\newcommand{\SO}{\mathsf{SO}}

\newcommand{\dvol}{d\mathrm{vol}}
\newcommand{\dsigma}{{d\sigma}}

\newcommand{\n}{\textbf{n}}

\newcommand{\id}{\mathrm{id}}
\newcommand{\pr}{\textrm{pr}}

\newcommand{\AK}{\mathcal{AK}}

\newcommand{\MM}{{M_1 \sqcup M_2}}

\DeclareMathOperator{\Rm}{Rm}
\DeclareMathOperator{\Ric}{Ric}

\DeclareMathOperator{\diam}{diam}

\DeclareMathOperator{\supp}{supp}

\DeclareMathOperator{\vol}{vol}

\numberwithin{equation}{section}

\begin{document}


\title[Weak Super Ricci Flow]{On Weak Super Ricci Flow through Neckpinch}


\address{  Sajjad Lakzian\\ \newline \phantom{S} School of Mathematics\\ Institute for Research in Fundamental Sciences (IPM)\\P.O. Box 19395-5746, Tehran, Iran }
\email{\href{mailto:slakzian@ipm.ir}{slakzian@ipm.ir}}

\address{Department of Mathematical Sciences\\Isfahan University of  Technology (IUT) \\Isfahan 8415683111, Iran }
\email{\href{mailto:slakzian@ipm.ir}{slakzian@ipm.ir} \\ \vskip 0.8cm}

\address{Michael Munn \\  \newline    \phantom{S}  Google, Greater New York City Area, USA}
\email{\href{mailto:mikemunn@gmail.com}{mikemunn@gmail.com}}



\maketitle

\begin{center}
	\textbf{\textit{ \footnotesize In honor of Mikhail Gromov on the occasion of his 75th birthday}}
\end{center}

\vspace{5mm}

\begin{center}
\begin{minipage}[t]{0.4\textwidth}
	\centering
\bf	{ \small Sajjad Lakzian}\\
\textit{\tiny IPM \& Isfahan University of Technology}\\
\textit{ \tiny Isfahan, Iran}\\
	
\end{minipage}
\begin{minipage}[t]{0.4\textwidth}
	\centering
	\bf{\small Michael Munn}\\
	\textit{\tiny Google}\\
	\textit{\tiny New York, USA}\\
\end{minipage}
\end{center}
\vspace{5mm}

\begin{abstract}
In this article, we study the Ricci flow neckpinch in the context of metric measure spaces. We introduce the notion of a Ricci flow metric measure spacetime and of a weak (refined) super Ricci flow associated to convex cost functions (cost functions which are increasing convex functions of the distance function). Our definition of a weak super Ricci flow is based on the coupled contraction property for suitably defined diffusions on maximal diffusion components. In our main theorem, we show that if a non-degenerate spherical neckpinch can be continued beyond the singular time by a smooth forward evolution then the corresponding Ricci flow metric measure spacetime through the singularity is a weak super Ricci flow for a (and therefore for all) convex cost functions if and only if the single point pinching phenomenon holds at singular times; i.e., if singularities form on a finite number of totally geodesic hypersurfaces of the form $\{x\} \times \sphere^n$. We also show the spacetime is a refined weak super Ricci flow if and only if the flow is a smooth Ricci flow with possibly singular final time.
\end{abstract}


\date{\today}


\tableofcontents
\section{Introduction}
Given a closed Riemannian manifold $(M^n,g)$, a smooth family of Riemannian metrics $g(t)$ on $M^n$ is said to evolve under the Ricci flow~\cite{Hamilton1982} provided
\begin{equation}
\label{EQN:RF}
\begin{cases}
	\dfrac{\partial}{\partial t} g(t) = -2\Ric (g(t));\\
	g(0) = g.
\end{cases}
\end{equation}
The uniqueness and short time existence of solutions to (\ref{EQN:RF}) was shown in~\cite{Hamilton1982} (see also~\cite{DeTurck}). When a finite time singularity occurs at some $T<\infty$, one gets 
\[
\lim_{t \nearrow T} \max_{x\in M^n}|\Rm(x,t)| = \infty.
\]
By the maximum principle, it follows that a singularity will develop in finite time once the scalar curvature becomes everywhere positive. A simple example of this can be seen for the canonical round unit sphere $\mathbb{S}^{n+1}$ which collapses to a point along the flow at $T=\frac{1}{2n}$.
\par While the shrinking sphere describes a global singularity, the neckpinch examples we are concerned with in this paper are local singularities; i.e., they occur on a compact subset of the manifold while keeping the volume positive. Intuitively, a manifold shaped like a barbell develops a finite-time local singularity as the neck part of the barbell contracts. The first rigorous examples of such a singularity were constructed in~\cite{AK1} where a class of rotationally symmetric initial metrics on $\mathbb{S}^{n+1}$, for $n\geq 2$ was found, which develop local Type-I neckpinch singularities through the Ricci flow (examples for non-compact manifolds did already exist~\cite{Simon-Pinch,Feldman-Ilmanen-Knopf}).
\subsection*{Motivation}
\par The general Sturmian theory about the behaviour of zeros of solutions to elliptic differential equations dates back to the 17th century; see Sturm~\cite{Sturmian}. In the context of rotationally symmetric Ricci flow, In~\cite{AK1}, authors applied a parabolic Sturm type result which was previously proven in~\cite{Angenent-88} to the evolution equation of $\psi_s$ where $\psi$ is the radius of the cross section spheres and $s$ the arclength parameter in radial direction. Their theorem results in finiteness and the decreasing behavior of the number of bumps and necks along smooth spherical Ricci flow and how and when these necks and bumps merge. With the additional assumptions that the metrics are reflection-invariant, the diameter remains bounded throughout the flow and the initial metric have a neck at $x=0$ and two bumps, it was shown that the neckpinch singularity occurs precisely at the equator, on the totally geodesic hypersurface $\{0\} \times \sphere^n$~\cite{AK2}. Differently put, approaching the singular time the number of necks decrease however at the singular time it might be the case that a neck gets flattened to an interval.~\cite{AK1} rules this out under the said extra conditions. Metaphorically speaking, the ``single point pinching phenomenon'' in the rotationally symmetric Ricci flow says ``a neck needs infinite room in order to spread onto an interval at the singular time''. This is very intuitive however not proven in general. We wish to study this question using methods other than the classical Ricci flow theory.
\subsection*{Goal}
\par The purpose of this paper is to put the single-point pinching question in the context of optimal transportation and provide a proof of the single-point pinching phenomenon in the setting where
\begin{itemize}
	\item Flow developes neckpinch singularity at finite time;
	\item There exists a continuation of the flow beyond the singular time by a smooth forward evolution of Ricci flow; see Definition~\ref{DEF:SFE};
	\item The resulting spacetime of product form has bounded diameter and the flow is a weak super Ricci flow corresponding to a time-dependent cost function $c_\tau(x,y) = h(d_\tau(x,y))$ for increasing convex $h$; see Definitions~\ref{DEF:WeakRicciFlow} and~\ref{DEF:RFSP-NP}.
\end{itemize}
\par Here, we differentiate between equatorial pinching and single point pinching phenomenon. For us, single point pinching refers to finite number of pinched totally geodesic hypersurfaces of the form $\{x\} \times \sphere^n$. 
\subsection*{Relevance of our set up}
The first hypothesis says our result is concerned only with finite time singularities which generally speaking are the more interesting ones. The second hypothesis is a now reasonable one to assume due to the work of~\cite{ACK} and in general, since by~\cite{Kleiner-Lott}, one expects singular Ricci flow to admit a smooth forward evolution of Ricci flow past the singular time and is expected to be unique in 3D. The third hypothesis which will be rigorously defined using optimal transport follows and is inspired by characterization of smooth super Ricci flows; see~\cite{McCann-Topping}) and~\cite{ACT} and is based on a theory of diffusions for time-dependent metric measure spaces using time dependent Cheeger-Dirichlet forms developed in~\cite{Kopfer-Sturm}. We will not assume any particular asymptotics or reflection symmetry which are necessary in existing proof of equatorial pinching~\cite{AK2}. 
\par The notion of Ricci flow metric measure spacetimes through pinch singularities defined here is different from Ricci flow spacetime in~\cite{Kleiner-Lott} in that we keep the singular set with its singular metric along the flow as part of the spacetime. In our setting, the metrics once become singular can only evolve in non-singular directions afterwards. This is also a intuitionally reasonable behavior to assume which is not covered in classical Ricci flow theory where the flow stops once a singularity is hit. 
\par Our definition of (refined) weak super Ricci flow with respect to a convex cost function $c$, which we will recall by $c-\mathcal{WSRF}$ for short, is based on a well-developed theory of diffusions for time-dependent metric measure spaces that has appeared in~\cite{Kopfer-Sturm} and has its roots in earlier works on time-dependent Dirichlet form theory. We will only be concerned with the time-dependent Cheeger-Dirichlet energy functional in order to construct diffusions. We note that the definition of weak super Ricci flow appearing here is different form the notion of weak super Ricci flow presented in~\cite{Sturm-SRF} in that we do not work with dynamic convexity of the Boltzman entropy as the definition of weak super Ricci flow and instead base our definition on the coupled contraction property for diffusions on suitable diffusion components which allows us to consider more general cost functions and more general (pinched) spaces.
\par Our approach can be applied to much more general cases than those considered in~\cite{AK1, AK2, ACK} and does not rely on the precise asymptotics obtained  and the mere existence of a smooth forward evolution suffices. This approach can be applied to other Ricci flow singularities which fit into a warped product regime. However, we will eventually only consider spherical rotationally symmetric flows that develop a non-degenerate neckpinch singularity at finite time and that can be extended beyond the singular time by a smooth forward evolution of Ricci flow; these metrics form our admissible initial data.
\par In relation to singular Ricci flows~\cite{Kleiner-Lott}, we note here that we think there is potential for applying these weak methods to singular Ricci flows in combination with $\mathcal{L}$-optimal transport characterization of super Ricci flows~\cite{Topping-L-OT} which is based on Perelman's reduced $\mathcal{L}$-distance. However, we will not explore this aspect in the present notes and only make some speculative remarks in~\textsection\ref{SEC:FUTURE}.
\subsection*{Quick tour}
\par The statement of our main theorem uses the notion of a ``rotationally symmetric spherical Ricci flow metric measure spacetime'' $\mathcal{X}$. The precise definition is given in~\textsection\ref{SEC:setup}; see Definitions~\ref{DEF:MMST} - \ref{DEF:RFMMST}. Roughly speaking, it refers to the spacetime constructed by considering a non-degenrate neckpinch on $\mathbb{S}^{n+1}$ and continued by smooth forward evolution on surviving (non-pinched) smooth pieces while keeping the singular points with their possibly evolving singular metrics. We will assume once the metric become singular at a point it can only evolve in its non-singular directions afterwards which means points in the interior of the singular set can not recover from singularities along the flow; see~\textsection\ref{SEC:setup} for precise definitions. 
\begin{definition}[admissible initial data]
	A rotationally symmetric Riemannian metric $g_0$ on $\sphere^{n+1}$ is said to be an admissible initial data, if the flow starting from $g_0$ develops a finite time non-degenerate neckpinch singularity at time $T$ which can be continued by a smooth forward evolution of Ricci flow (of finite diameter) beyond the singular time. We denote by $\mathcal{ADM}$, the set of all admissible initial data. 
\end{definition}

\begin{customthm}{1}\label{THM:MAIN}
Suppose $g_0 \in \mathcal{ADM}$ and let for some $\epsilon > 0$, $\left(  \sphere^{n+1} \times [0,T + \epsilon),  \mathbf{t}, \partial_t, g \right)$ be a resulting rotationally symmetric spherical Ricci flow pseudo-metric measure spacetime through singularity as in Definition~\ref{DEF:RFSP-NP}. Then, the single-point pinching holds if and only if the spacetime is a {\bf weak super Ricci flow} with respect to a/ any time dependent cost function $c_\tau$ of the form $c_\tau(x,y) = h\left( d_\tau(x,y) \right)$ where $h$ is strictly convex or the identity; see ~Definition~\ref{DEF:WeakRicciFlow}. Furthermore, $\left(  \sphere^{n+1} \times [0,\bar{T}],  \mathbf{t}, \partial_t, g \right)$ is a {\bf refined weak super Ricci flow} associated to cost $c_\tau$ if and only if the flow is smooth except possibly at time $\bar{T}$.
\end{customthm}
As was previously mentioned, the equatorial pinching result assumes reflection symmetry of the metric in addition to a diameter bound; see~\cite[\textsection10]{AK1}. In~\cite{AK2}, it was shown that reflection symmetry in fact implies the diameter bound the proof of which relies on careful analysis and detailed computations arising from the imposed evolution equations on the profile warping function $\psi(r,t)$. In~\textsection\ref{SEC:PROOF} we provide the proof of our main Theorem~\ref{THM:MAIN} which puts the ``one-point pinching'' result in a rather general context. Our method of proof involves techniques of optimal transportation and, as such, allows us to avoid the requirement of reflection symmetry and needing precise asymptotics, though we still require the condition on the diameter bound. 
\subsection*{Ideas of the proof}
 In the setting that we consider in this article, the proof of single point pinching relies on two important and yet intuitive observation. One is the interior of the singular set can not conduct diffusions and so diffusions starting in smooth parts stay in smooth parts. Second is single point pinches resulting in double points after the singular time will force the dynamic heat kernel to fail to be Holder continuous. 
 \par We then define diffusion components to be those which support a diffusion theory for time-dependent Cheeger-Dirichlet energy in which the dynamic heat kernel is Holder continuous in the support of the measures. These spaces can not contain double points. We will also consider rough diffusion components to be those which support a diffusion theory for Cheeger-Dirichlet spaces with a dynamic heat kernel which is allowed to be only measurable. These components can have double points in their interior. 
 \par A weak super Ricci flow (resp. refined weak super Ricci flow) is then a spacetime which can in almost everywhere sense be covered by maximal diffusion components (resp. maximal refined diffusion components) such that on any of the maximal diffusion components the coupled contraction property associated to cost $c$ holds. Of course with this definition, the weaker diffusions we allow, the more regular flows we get. So rough diffusions will lead to refined weak super Ricci flows.
\par Under the assumption that the spacetime is a weak super Ricci flow, a maximal diffusion component can only a priori contain only a neckpinch on an interval of positive length. One can make the rate of change of the total optimal cost of transportation between diffusions arbitrarily large by exploiting the inifinite propagation speed of diffusions under Ricci flow. This in turn tells us that the rate of change of the length of singular neck is unbounded to accommodate the requirement that the flow is a weak super Ricci flow. The rotational symmetry allows us to convert the computation of total cost to a 1D problem for which we can use precise formulas that exist for convex costs.
\par Under the assumption that the spacetime is a refined weak super Ricci flow, maximal refined diffusion components {\bf can} contain both interval and single point pinching and using similar arguments as discussed, one can show the existence of either types of singularities (except at the final time) will defy coupled contraction property. 
\subsection*{Organization of materials}
 In~\textsection\ref{SEC:PRELIMS} we will quickly touch upon relevant background material. In particular, we begin in~\textsection\ref{Neckpinch prelim} by discussing the formation of neckpinch singularities through the Ricci flow. We in~\textsection~\ref{SEC:SFERF} and \textsection\ref{SEC:KLSRF} briefly touch upon smooth forward evolution of Rici flow and of singular Ricci flows. In~\textsection\ref{SEC:OTPrelim} we recall the basic ideas of optimal transportation and especially 1D formulas. Of particular interest is~\textsection\ref{Sec:Existing} in that it contains information about the diffusion theory we will employ. ~\textsection\ref{SEC:OTRot} is devoted to reducing the computation of optimal cost of transport for spatially uniform probability measures to a computation on the base space when the base space is $c$-fine. In~\textsection\ref{SEC:setup}, we present the definitions and constructions of Ricci flow spacetimes and that of weak diffusions and weak super Ricci flows. In particular, in~\textsection\ref{Sec:diffusions} we provide a slight variation of diffusion theory~\cite{Kopfer-Sturm}, based of which we will define two types of weak super Ricci flows in~\textsection\ref{Sec:WSRF}. Finally, in~\textsection\ref{SEC:PROOFofTHM1.1}, we will provide the proof of our main result and will briefly mention how these ideas can be generalized to address more general neckpinch singularities. 
\begin{ack} 
\small	The authors would like to thank C. Sormani who originally encouraged the study of Ricci flow neckpinch using metric geometry. The authors would also like to thank the Hausdorff Research Institute where some of this work was completed during a Trimester program in Optimal Transport.
\par SL is grateful to D. Knopf for many insightful conversations and his hospitality and to K. T. Sturm for valuable mentoring during SL's Hausdorff postdoctoral fellowship. This work initiated when SL was supported by the NSF under DMS-0932078 000 while in residence at the Mathematical Science Research Institute during the Fall of 2013, continued when SL was supported by the Hausdorff Center for Mathematics in Bonn. At the time of completion of the project, SL is supported by Isfahan University of Technology and the School of Mathematics, Institute for Research in Fundamental Sciences (IPM) in Tehran (and Isfahan), Iran.
\normalsize
\end{ack}
\addtocontents{toc}{\setcounter{tocdepth}{2}}
\section{Background and preliminaries}
\label{SEC:PRELIMS}
\subsection{Neckpinch singularities of the Ricci flow} 
\label{Neckpinch prelim}
Singularities of the Ricci flow can be classified according to how fast they are formed. A solution $(M^n, g(t))$ to (\ref{EQN:RF}) develops a Type I, or rapidly forming, singularity at $T < \infty$, if 
\[
\sup_{M\times [0, T)} (T-t)\left|\Rm (\,\cdot\,, t)\right| < +\infty.
\]
Such singularities arise for compact 3-manifolds with positive Ricci curvature~\cite{Hamilton1982}. Indeed, any compact manifold in any dimension with positive curvature operator must develop a Type I singularity in finite time~\cite{BW2008}.
\par A solution $(M^{n+1}, g(t))$ of the Ricci flow is said to develop a \textit{neckpinch singularity} at some time $T < \infty$ through the flow by pinching an almost round cylindrical neck. More precisely, there exists a time-dependent family of proper open subsets $N(t) \subset M^{n+1}$ and diffeomorphisms $\phi_t: \R \times \sphere^n \to N(t)$ such that $g(t)$ remains regular on $M^{n+1} \setminus N(t)$ and the pullback $\phi_t^*\bigl(\left.g(t)\right|_{N(t)}\big)$ on $\R \times \sphere^n$ approaches the ``shrinking cylinder'' soliton metric
\[
ds^2 + 2(n-1)(T-t) g_{\text{can}}
\]
in $\mathcal{C}^{\infty}_{loc}$ as $t \nearrow T$ where $g_{\text{can}}$ denotes the canonical round metric on the unit sphere $\sphere^n(1)$. The neckpinch is called a non-degenerate neckpinch if the shrinking cylinder soliton is the only possible blow up limit as opposed to the degenerate examples which could lead to other solitons (such as a Bryant soliton) as a blow up limit~\cite{AIK-degenrate}. 
\par Following~\cite{AK1}, consider $\sphere^{n+1}$ and remove the poles $P_{\pm}$ to identify $\sphere^{n+1} \setminus{P_{\pm}}$ with $(-1,1) \times \sphere^n$. An SO($n+1$)-invariant metric on $\mathbb{S}^{n+1}$ can be written as 
\begin{equation}
\label{EQN:InitialRotSymg} 
g = \phi(x)^2 (dx)^2 + \psi(x)^2 g_{\text{can}}
\end{equation}
where $x \in (-1,1)$. Letting $r$ denote the distance to the equator given by 
\[ 
\label{EQN:rDefn} 
r(x) = \int_0^x \phi(y) dy;
\]
one can rewrite (\ref{EQN:InitialRotSymg}) more geometrically as a warped product 
\[
g = (dr)^2 + \psi(r)^2 g_{\text{can}}.
\]
It is customary to switch between variables $x$ and $r$ when necessary in computations and the derivatives with respect to these variables are related by 
$
\frac{\partial}{\partial r} = \frac{1}{\phi(x)} \frac{\partial }{\partial x}.
$
To ensure smoothness of the metric at the poles, one requires $\lim_{x\to \pm 1} \psi_r = \mp 1$ and $\phi/(r_{\pm}-r)$ is a smooth even function of $r_{\pm}-r$ where $r_{\pm} :=r(\pm1)$. This assumption of rotational symmetry on the metric allows for a simplification of the full Ricci flow system from a non-linear PDE into a system of scalar parabolic PDE in one space dimension.
\begin{equation}
\label{phipsi evolve}
\begin{cases}
(\psi_1)_t = (\psi_1)_{rr} - (n-1)\dfrac{1-(\psi_1)_r^2}{\psi_1}; \\
(\phi_1)_t = \dfrac{n(\psi_1)_{rr}}{\psi_1}\phi_1;
\end{cases}
\end{equation}
furthermore, the Ricci curvature is given by
\begin{equation}
\label{Ricci}
\Ric = -n\frac{(\psi)_{rr}}{\psi} dr^2 + \left(-\psi_1(\psi)_{rr} - (n-1(\psi)_r^2 + n-1 \right) g_{\text{can}};
\end{equation} 
e.g.~\cite{AK1, AK2}.
\par In~\cite{AK1}, the authors establish existence of Type I neckpinch singularities for an open set of initial $\SO(n+1)$-invariant metrics on $\sphere^{n+1}$ satisfying the properties:
\begin{itemize}
	\item positive scalar curvature on all of $\sphere^{n+1}$ and positive Ricci curvature on the ``polar caps'', i.e., the part from the pole to the nearest ``bump'';
	\item positive sectional curvature on the planes tangential to $\{x\} \times \sphere^n$;
	\item ``sufficiently pinched'' necks; i.e., the minimum radius should be sufficiently small relative to the maximum radius; c.f.  \textsection8 of~\cite{AK1}.
\end{itemize}
Throughout this paper, let $\AK$ denote this open subset of initial metrics on $\sphere^{n+1}$ defined above and let $\AK_0 \subset \AK$ denote those metrics in $\AK$ which are also reflection symmetric, i.e., $\psi(r, 0) = \psi(-r,0)$ and which furthermore have a ``single'' symmetric neck at $x=0$ and two bumps. The set $\tilde{\AK}$ is those initial conditions that satisfy the asymptotic conditions of~\cite[Table 1]{ACK} at the singular time. So clearly the inclusion relations
\[
       \AK_0 \subset \AK \subseteq \tilde{\AK} \subseteq \mathcal{ADM}
\]
hold.
\par Given an initial metric $g_0 \in \AK$, the solution $(\sphere^{n+1}, g(t))$ of the Ricci flow becomes singular, developing a neckpinch singularity, at some $T < \infty$. Furthermore, provided $g_0 \in \AK_0$, its diameter remains bounded for all $t \in [0, T)$ and the singularity occurs only on the totally geodesic hypersurface $\{0\} \times \sphere^n$; see~\cite[\textsection10]{AK1} and~\cite[Lemma 2]{AK2}.
\par In~\cite{AK}, equatorial pinching is proven assuming restrictive conditions that the metric $g(t)$ has at least two bumps, i.e., local maxima of $x \mapsto \psi(x,t)$, and denote the locations of those bumps by $x_-(t)$ and $x_+(t)$ for the left and right bump (resp.) and a neck at $x=0$. So it is not really clear (even though highly expected) how the single-pint pinching can be proven for general neckpinches. To show that the singularity occurs only on $\{0\} \times \sphere^{n}$, they show that $\psi(r, T) >0$ for $r>0$. Their proof of this one-point pinching requires delicate analysis and construction of a family of subsolutions for $\psi_r$ along the flow.
\par An important consequence of the Sturmian result obtained in~\cite[Lemma 5.5]{AK1} is that even-though there is the possibility that a neck could spread over an interval at the singular time, the number of connected components of the singular set at the first singular time has to be finite. So the (closed) singular set is apriori a finite disjoint union of closed intervals and points.
\par It is noteworthy that in~\cite{AK2}, precise asymptotics ar eobtained for such flows going into the non-degenrate singularity. These asymtotics are important in that, in conjunction with~\cite{ACK}, they provide sufficient conditions for a smooth forward evolution of Ricci flow to exist. 
\subsection{Smooth forward evolution}
\label{SEC:SFERF}
In~\cite{ACK}, the authors extend the spherical neckpinch past the singular time by constructing smooth forward evolutions of the Ricci flow starting from initial singular metrics which arise from these rotationally symmetric neck pinches on $\sphere^{n+1}$ described above. To do so requires a careful limiting argument and precise control on the asymptotic profile of the singularity as it emerges from the neckpinch singularity (solution is constructed in pieces by the method of formal matched asymptotics). By passing to the limit of a sequence of these Ricci flows with surgery, effectively the surgery is performed at scale zero. Following~\cite{ACK}, define
\begin{definition}[smooth forward evolution]\label{DEF:SFE}
	A smooth complete solution $(M^n, g(t))$ on some interval $t \in (T, T')$ of the Ricci flow is called a {\bf forward evolution} of a singular Riemannian metric $g(T)$ on $M^n$ if as $t \searrow T$, $g(t)$ converges to $g(T)$ in $C^\infty_{loc}(\mathcal{R})$, where $\mathcal{R}$ is the (open) regular set of $g(T)$.
\end{definition}
\par Let $g_0$ denote a singular Riemannian metric on $\sphere^{n+1}$, for $n\geq 2$, arising from the limit as $t \nearrow T$ of a rotationally symmetric neckpinch forming at time $T <\infty$. Given that this singular metric satisfies certain asymptotics (see~\cite[Page 6]{ACK}) in particular if the pinching is already at the equator, there exists a complete smooth forward evolution $(\sphere^{n+1}, g(t))$ for $T < t < T'$ of $g_0$ by the Ricci flow. Any such smooth forward evolution is compact and satisfies a unique asymptotic profile as it emerges from the singularity and is expected to be unique; see~\cite[Theorem 1]{ACK}.
\par Viewed together,~\cite{AK1, ACK,AK2} provide a framework for developing the notion of a ``canonically defined Ricci flow through singularities'' as conjectured by Perelman in~\cite{Perelman-RFWS}, albeit in this rather restricted context of rotationally symmetric non-degenerate neckpinch singularities which arise for these particular initial metrics on $\mathbb{S}^{n+1}$, i.e., for $g_0 \in \AK_0$. Such a construction for more general singularities remains a very difficult issue. Even construction of precise examples are hard; see~\cite{IKS} for first precise examples. Prior to~\cite{ACK}, continuing a solution of the Ricci flow past a singular time $T < \infty$ required surgery (and throwing away some parts of the flow containing singular points) and a series of carefully made choices so that certain crucial estimates remain bounded through the flow. 
\par Ideally, a complete canonical Ricci flow through singularities would avoid these arbitrary choices and would be broad enough to address all types of singularities that arise in the Ricci flow. We refer the reader to~\cite{Hamilton-Four-Manifolds, Perelman-RFWS, Perelman-entropy, Perelman-RFsolns}. Recent work~\cite{Kleiner-Lott} on singular Ricci flows answers Perelman's question in three-dimensions in rather full generality and due to its importance and relevance, it will be briefly touched upon in the next section.
\par Existence of smooth forward evolution of Ricci flow for singular metrics not satisfying the said asymptotics does not follow from~\cite{ACK} however, it is implied by~\cite{Kleiner-Lott}. In our study we will consider singular metrics that admit a smooth forward evolution based on Definition~\ref{DEF:SFE}. This will be elaborated on in~\textsection\ref{SEC:setup}. 

\subsection{The singular Ricci flow}
\label{SEC:KLSRF}
The Perelman's question ``Is there a canonical Ricci flow through Singularities?'' has motivated some important work in the field and as such is the theory of singular Ricci flows introduced in~\cite{Kleiner-Lott}. One way singular Ricci flow is important to us is that heuristically speaking, the existence of singular Ricci flow implies the existence of a smooth forward evolution in three dimensions although we are not proving this here nor use it and instead will restrict ourselves to initial metrics for which a smooth forward evolution beyond the singularity exists. Exploring the optimal transport in this spacetime would potentially lead to some interesting results. We will make further remarks in~\textsection\ref{SEC:FUTURE} as to how our ideas could be adapted to singular Ricci flows. 
\par A Ricci flow spacetime in the sense of~\cite{Kleiner-Lott} is a quadruple $\left( \mathcal{M}, \mathbf{t}, \partial_t, g \right)$ consisting of a possibly incomplete four dimensional Riemannian manifold $\mathcal{M}$, a submersive time function  $\mathbf{t}: \mathcal{M} \to  I \subset \R$, a time vector field $\partial_t$ which is the gradient of time function and Riemannian metric $g$. The Ricci flow equation is captured by requiring $\mathcal{L}_{\partial_t} g = - 2 \Ric(g)$. A singular Ricci flow is then defined to be an eternal four dimensional Ricci flow spacetime where the $0$-th time slice is a normalized three dimensional Riemannian manifold, satisfies Hamilton-Ivey curvature pinching condition and is $\varkappa$-non-collapsed below the scale $\epsilon$ and satisfies $\rho$- canonical neighborhood assumption for a fixed $\epsilon$ and decreasing functions $\varkappa, \rho: \R_{\ge0} \to \R_{\ge0}$~\cite{Kleiner-Lott}.
\par Any compact normalized three dimensional Riemannian manifold is the $0$-th time slice of a singular Ricci flow; see~\cite[Theorem 1.3]{Kleiner-Lott}. Another striking result is the uniqueness of such singular Ricci flows obtained as limits of Ricci flows with surgery which is proven in~\cite{Bamler-Kleiner}. These two theorems together affirmatively answers Perelman's question in the sense that it shows there is a way to uniquely flow a three dimensional manifold through singularities. 
\par This existence and uniqueness theory for singular Ricci flows in three dimensions substantiate our hypothesis of existence of smooth forward evolution at least for three dimensional Ricci flow neckpinch.
\subsection{Various weak notions of Ricci flow}
By now there are various ways to make sense of Ricci flow or supersolutions thereof (super Ricci flows) starting from initial data which are less regular than $C^2$ or are even non-Riemannian (in a broad sense when the Sobolev space $W^{1,2}$ is not Hilbert). We have already touched upon smooth forward evolution and  singular Ricci flows. In later sections we will also briefly review characterization of smooth super Ricci flows using optimal transportation \cite{McCann-Topping, ACT, Sturm-SRF, Kopfer-Sturm} which has motivated us for definition of weak super Ricci flows associated to convex cost functions. We deemed fit to include a brief and non-exhaustive list of some other progresses in this direction. 
\par In the setting of Finsler structures, Finsler-Ricci flow has been introduced in~\cite{Bao} and differential Harnack estimates for positive heat solutions under Finsler-Ricci flow can be found in~\cite{Lakzian-DHE-FRF} for vanishing $S$-curvature along the flow. 
\par Concerning K\"{a}hler-Ricci flow, weak K\"{a}hler-Ricci flow out of initial K\"{a}hler currents with $C^{1,1}$ potential is introduced in~\cite{CTZ}. K\"{a}hler-Ricci flow thorough singularities using divisorial contractions and flips was considered in~\cite{Song-Tian}. See~\cite{Zhou,GZ, EGZ} for related work in this direction.
\par In the setting of metric measure spaces, it is noteworthy that in~\cite{Haslhofer-Naber}, smooth Ricci flow has been characterized in terms of gradient estimates for heat flow on the path space using methods of stochastic analysis that can potentially be extended to RCD spaces; see also~\cite{ACT} for related work on path space. A rather full theory of weak super Ricic flows has been developed in~\cite{Sturm-SRF} which is based on dynamic convexity of the Boltzmann entropy. These weak flows can be characterized in terms of couple contraction property for diffusions when there is no pinching in the metric (by which, we mean the time slices are RCD spaces and the time dependent metric and measure satisfy some regularity assumptions); see~\cite{Kopfer-Sturm, Kopfer}. In particular, the coupled contractio property proven in~\cite{Kopfer-Sturm} is not directly applicable to the Ricci flow through a neckpinch singularity. Our notion of weak super Ricci flow associated to a convex cost $c$ is based on weak diffusions for a time-dependent metric measure space constructed in~\cite{Kopfer-Sturm} on so called maximal diffusion components.
\par Regarding similar intrinsic flows on metric measure spaces, an intrinsic flow tangent to Ricic flow has been introduced in~\cite{Gigli-Mantegazza} via heat kernel methods. The regularization effect of this flow has been studied in~\cite{BLM, Erbar-Juillet}.

\subsection{Optimal transportation}
\label{SEC:OTPrelim}
Let $(X,d)$ be a compact metric space and consider the space of Borel probability measures on $X$ denoted $\mathscr{P}(X)$. Given two probability measures $\mu_1, \mu_2 \in \mathscr{P}(X)$ the optimal total cost of transporting $\mu_1$ to $\mu_2$ with cost function $c(x,y)$ is given by
\begin{equation}
	\label{DEF:totalcost}
	\mathcal{T}_c(\mu_1, \mu_2) = \left[\inf_{\pi \in \Gamma(\mu_1,\mu_2)} \int_{X^2} c(x,y) d\pi(x,y)\right],
\end{equation}
where the infimum is taken over the space $\Gamma(\mu_1, \mu_2)$ of all  joint probability measures $\pi$ on $X^2$ which have marginals $\mu_1$ and $\mu_2$; i.e., for projections $\pr_1, \pr_2: X^2 \to X$ onto the first and second factors (resp.), one has
\[
({\pr_1})_\# \pi = \mu_1 \quad \text{ and } \quad  ({\pr_2})_\# \pi = \mu_2.
\]
Any such probability measure $\pi\in \Gamma(\mu_1, \mu_2)$ is called a {\it transference plan} between $\mu_1$ and $\mu_2$ and when $\pi$ realizes the infimum in (\ref{DEF:totalcost}) we say $\pi$ is an {\it optimal transference plan}. Monge-Kantrovitch's problem is solving for an optimal $\pi$ and the classical Monge's problem is solving for an optimal plan induced by a transport map $T: X \to X$ i.e. when
\[
\pi = \left( \id \times T  \right)_\sharp \mu_1. 
\] 
\par A particular case is to use the $L^p$ cost function $c(x,y) = d^p(x,y)$. The $L^p$-Wasserstein distance between $\mu_1$ and $\mu_2$  is given by 
\begin{equation}
W_p(\mu_1, \mu_2) = \left[\inf_{\pi \in \Gamma(\mu_1,\mu_2)} \int_{X^2} d(x,y)^p d\pi(x,y)\right]^{1/p}, \quad \text{ for } p \geq 1. \notag
\end{equation}
\subsubsection{Optimal transport on $\mathbb{R}$} In the one dimensional case, the total cost of transport between two absolutely continuous measures with cost functions of the form $c(x,y) = h\left( \left|  x - y  \right| \right)$ with $f$ convex, can be computed directly from the cumulative distribution functions. The 1D case is important for us since as we will show in~\textsection\ref{SEC:OTRot}, computing the optimal total transport cost between two spatially uniform measures in the rotationally symmetric case reduces to the 1D problem of computing optimal total transport cost between the projection of the said measures to the real line. The material in this section are by now standard; e.g. ~\cite{RR, Villani-TOT, Santambrogio}.
\par Any probability measure $\mu$ on $\mathbb{R}$ can be represented by its cumulative distribution function
\[
F(x) = \int_{-\infty}^x d\mu = \mu\left[ (-\infty, x)\right].
\]
From basic probability theory, it follows that $F$ is right-continuous, non-decreasing, $F(-\infty)=0$, and $F(+\infty)=1$. The generalized inverse of $F$ on $[0,1]$ is defined by
\[F^{-1}(t) = \inf\{x \in \mathbb{R} ~:~ F(x) >t\}.\]
\par Let the cost function be a convex function of distance. For two probability measures $\mu_1$ and $\mu_2$ on $\R$ with respective cumulative distributions $F$ and $G$, the total cost of transport between $\mu_1$ and $\mu_2$ is given in terms of the cumulative distributions by the formula
\begin{equation}
\label{1D-totalcost} 
\mathcal{T}_c(\mu_1, \mu_2) = \int_0^1 h\left(  \left| F^{-1}(t) - G ^{-1}(t) \right| \right) dt;
\end{equation}
furthermore, for the linear cost, Fubini's theorem implies
\begin{equation}
\label{W1 vs L1 dist} 
W_1(\mu_1,\mu_2) = \int_0^1 |F^{-1}(t) - G ^{-1}(t)| dt=  \int_{\R} |F(x) - G (x)| dx.
\end{equation}
See~\cite[Theorem 3.7.1]{RR} and~\cite[Theorem 2.18]{Villani-TOT}; c.f. Santambrogio~\cite[Proposition 2.17]{Santambrogio}; also see Figure~\ref{fig:distributions1}. 
\subsection{Ricci flow and optimal transport} 
\label{SEC:RFandOT}
\subsubsection{Characterization of smooth super Ricci flows via coupled contraction for diffusions}
The importance of considering super Ricci flows was first highlighted~\cite{Topping-McCann} and was tied to the theory of optimal transportation. Also see the other related work~\cite{Lott-OT, Topping-L-OT, Topping-FoundationsOT, ACT} and the more recent works~\cite{Sturm-SRF, Kopfer-Sturm}.
\par Given a solution to (\ref{EQN:RF}) on some time interval $t \in [0, T]$, let $\tau := T-t$ denote the backwards time  parameter. Note that  $\del/ \del \tau = - \del / \del t$. A Ricci flow parameterized in backward time (or backward Ricci flow) is $g(\tau)$ where
\begin{equation}
\label{EQN:bRF}
\frac{\partial g}{\partial \tau} = 2\Ric(g(\tau)). 
\end{equation}
Subsolutions of (\ref{EQN:bRF}) (or equivalently supersolutions of Ricci flow equation) are called super Ricci flows and satisfy
\[\frac{\partial g}{\partial \tau} \le 2 \Ric (g(\tau)).\]
\par A family of smooth measures $\mu(\tau)$ whose element is a smooth $n-$form $\omega(\tau)$, is called a diffusion if 
\[
\frac{\partial u}{ \partial \tau } \omega(\tau) = \Delta_{g(\tau)}  \omega(\tau);
\]
equivalently, the density function satisfies the conjugate heat equation
\begin{equation}
\label{EQN:ConjHeatEqn}
\frac{\partial u}{ \partial \tau } =  \Delta_{g(\tau)} u - R_{g(\tau)}u.
\end{equation}
where $R_{g(\tau)}$ denotes the scalar curvature~\cite{Topping-McCann}.
\par It is shown in~\cite{Topping-McCann} that a smooth one parameter family of metrics is a super Ricci flow if and only if the $L^2$-Wasserstein distance between two diffusions is non-increasing in $\tau \in (a,b)$; namely the so-called dynamic coupled contraction property is equivalent to being a super Ricci flow. The same also holds using the $L^1$-Wasserstein distance; c.f.~\cite{Topping-FoundationsOT}. 
\par The dynamic coupled contraction for general cost functions (cost functions which are increasing in distance) has been shown to holds under
\[
\frac{\partial g}{\partial \tau} \le 2 \Ric (g(\tau)) + \nabla_\tau^2f;
\]
so in particular, dynamic coupled contraction holds for supper Ricci flows by setting the potential $f$ equal to $0$; see~\cite[Theorem 4.1]{ACT}.
\par The proof of coupled contraction in~\cite{ACT} does not provide \emph{equivalence} even for the $L^p$ cost function with $p\neq 1,2$. So it is still unknown to us that for what cost functions does the contraction property coincide with the flow being a super Ricci flow. Our guess is this should hold true at least for costs that are strictly convex functions of distance. Nevertheless, the results mentioned thus far motivate us to define weak super Ricci flows corresponding to a convex cost function $c_\tau = h \circ d_\tau$ by the dynamic coupled contraction property. We will pursue this approach in~\textsection\ref{Sec:WSRF}. 
\par We note that the notion and exmaples of super Ricci flows on a metric space consisting of disjoint union of Riemannian manifolds has been explored in~\cite{Lakzian-Munn} as the first smooth examples yet outside the standard realm of smooth Ricci flows. 
\subsection{A closer look at relevant existing theories}
\label{Sec:Existing}
A theory of ($N$-) super Ricci flows for time-dependent metric measure spaces has been recently developed in~\cite{Sturm-SRF}. According to~\cite{Sturm-SRF}, an ($N$-) super Ricci flow is a time-dependent metric measure space (with fix underlying set) for which the Boltzmann entropy is strongly dynamical ($N$-) convex; see~\cite{Sturm-SRF} for further details. 
\par The interplay between dynamical convexity of the Boltzmann entropy and of coupled contraction property of the $L^2$-Wasserstein distances for diffusions has been explored in~\cite{Kopfer-Sturm} much in the spirit of~\cite{McCann-Topping} in the Riemannian setting yet with much more added difficulty of having to establish a theory of gradient flow of time-dependent Dirichlet energy and of Entropy. A unified approach to gradient flow theory in the static setting had been explored for rather general spaces in~\cite{AGS-GF} though, time-dependent spaces require more delicate analysis as carried out in~\cite{Kopfer-Sturm}. 
\par In the static case, a regular symmetric strongly local closed Dirichlet form gives rise to a diffusion semigroup given by a kernel; see the classic reference~\cite{FOT}. In relation to the geometry of a metric measure spaces, this approach was taken on, in e.g.~\cite{Sturm-HK}. In the dynamic setting, a time-dependent Dirichlet form can still be used to construct diffusions if we assume ellipticity with respect to a fixed Dirichlet form and if we assume square field operators satisfy the chain rule; see~\cite{Sturm-DirII, Oshima2004, Lierl20}. 
\par For a time-dependent metric measure space,~\cite{Kopfer-Sturm} has brought all these various aspects together and presented a fine theory of diffuions linked to ($N$-) super Ricci flows introduced in~\cite{Sturm-SRF}.~\cite{Kopfer-Sturm} provides a theory of what we will refer to as $L^2$-weak super Ricci flow. Dynamic coupled contraction for ($N$-) super Ricci flows has been verified under strong regularity assumptions on the time slices such as curvature bounds. These assumptions fail to hold for flows going through non-degenerate neckpinches. 
\par Following~\cite{Kopfer-Sturm}, let $\left( X, d_t, m_t  \right)$ be a family of Polish compact metric measure spaces where the Borel measures $m_\tau$ at different times are mutually absolutely continuous with logarithmic densities (with respect to a fixed measure) that are Lipschitz in time and where the distances satisfy the non-collapsing property
\[
\left|  \ln \frac{d_t(x,y)}{d_s(x,y)}  \right| \le L |t-s|.
\]
which they refer to as the log-Lipschitz condition. Under these conditions and given a time dependent family of Dirichlet forms whose square field operators satisfy diffusion property and uniform ellipticity with respect to a given background strongly local regular Dirichlet form, and if the square field operators applied to logarithmic densities are uniformly bounded, solutions to dynamic heat and the dynamic conjugate heat equations (defined weakly in terms of the Dirichlet form) exist and are unique; see~\cite{Kopfer-Sturm} for details.
\par In our setting, the flow undergoes a neckpinch singularity which causes the loss of doubling condition and although the time slices still remain infinitesimally Hilbertian (e.g. ~\cite{IH-property}), they will not satisfy any curvature bounds. Indeed the time slices are not essentially non-branching anymore and furthermore at the singular time, there is a cusp singularity so the volume doubling fails. Therefore, the work in~\cite{Kopfer-Sturm} does not provide equivalence between dynamic coupled contraction and ($N$-) super Ricci flows. On the other hand, since the static convexity fails, we expect $N$-convexity of Boltzman entropy to fail for these spaces and yet it only makes sense for such a flow to be a super Ricci flow in some sense. All this facts prompt us to instead directly use the dynamic coupled contraction as the definition of weak super Ricci flow. This is the approach that we will take on in this article. 

\subsection{Warped product of metric spaces}
\label{Sec:Warped}
We will need, by now standard, generalization of warped products to the metric space setting along with the characterization of geodesics therein~\cite{Chen1999}. Let $(X, d_X)$ and $(Y, d_Y)$ are two metric spaces and  $f: X \to \R^+$ is a continuous function on $X$. For two points $a,b \in X \times Y$, according to~\cite{Chen1999}, one defines
\[
d(a,b) := \inf_{\gamma} \{L_f(\gamma) ~:~ \gamma \text{ is a curve from } x \text{ to } y\}
\]
where $L_f(\gamma)$ denotes the length of the curve $\gamma(s) := (\alpha(s), \beta(s)) \in M \times N$ given by 
\[
L_f(\gamma) : = \lim_{\tau} \sum^n_{i=1} \Big(  d_X^2(\alpha(t_{i-1}),\alpha(t_i)) + f^2(\alpha(t_{i-1}))d_Y^2(\beta(t_{i-1}), \beta(t_i))  \Big)^{\frac{1}{2}},
\]
and the limit is taken with respect to the refinement ordering of partitions $\tau$ of $[0,1]$ denoted by $\tau : 0 = t_0 < t_1 < \cdots < t_n = 1$.
$X\times Y$ equipped with this metric $d$ is called the {\it warped product of $X$ and $Y$ with warping function $f$} and is denoted $(X\times_f Y, d)$. $d$ is indeed a distance metric on $X \times Y$~\cite{Chen1999}.
\par With $(X, d_X), (Y,d_Y)$ and $f$ as above, if $X$ is a complete, locally compact metric space and $Y$ is a geodesic space, then for any $a,b \in (X \times_f Y, d)$ there exists a geodesic joining $a$ to $b$~\cite{Chen1999}.
\par Geodesics in $X$ lift horizontally to geodesics in $(X \times_f Y,d)$. Furthermore, if $(\alpha, \beta)$ is a geodesic in $(X \times_f Y, d)$ then $\beta$ is a geodesic in $Y$~\cite{Chen1999}.
\section{Optimal transport in warped product spaces}
\label{SEC:OTRot}
Here, we wish to examine the optimal transport problem for two spatially uniform measures $\mu_1, \mu_2$ on a given metric space which is of the warped product form. This will be used to reduce the computation of optimal total transport cost of such measures in a warped product space to computation of optimal total transport cost between the projection of the measures to the base space. 
\subsection{Spatially uniform measures}
In our study, we will need to be able to handle collapsed spaces such as time slices in spherical Ricci flow metric measure spacetimes through neckpinch singularities. So we will be concerned with warped products in which the warping function is only non-negative. Let $\left(  X, d_X, m_X  \right)$ and $\left( Y, d_Y, m_Y \right)$ be two geodesic metric measure spaces. Let $f: X \to \R^{\ge0}$ be a non-negative warping function. The construction in~\textsection\ref{Sec:Warped} can be carried out to obtain, this time, a semi metric space
\[
\tilde{Z} = X \times_f Y
\]
equipped with an intrinsic semi distance $\tilde{d}$ computed as in~\ref{Sec:Warped}. Notice that $\tilde{d}$ is a semi metric since when $f(x) = 0$, 
\[
\tilde{d} \left( (x,p), (x,q)  \right) = 0, \quad \forall p,q \in Y.
\]
Let $Z$ be the quotient space
\[
Z := \frac{\tilde{Z}}{\sim}, \quad (x,p) \sim (x,q)\quad  \mathrm{iff}\quad  f(x) =0 \quad \text{or} \quad p=q.
\]
Then, $\tilde{d}$ reduces to a distance $d$ on the quotient space $Z$ turning it into an intrinsic geodesic metric space. When it comes to picking a background measure on $Z$, one has variety of choices. Even though the optimal transport problem between two measures can be defined regardless of what the background measure is, many nice properties of the transport can be deduced from the properties of the background measure when the two transported measures are absolutely continuous with respect to the background measure. For example the background measure will obviously affect the curvature bounds for the warped product space which in turn can be used to show existence of optimal maps; see~\textsection\ref{Sec:finespaces}. Any reasonable and geometrically meaningful background measure should be absolutely continuous with respect to the product measure on parts where the warping function is positive so, for us this is the class of admissible background measures for which we will state our results. 
\begin{definition}
Let $Z$ be as in above and $P: \tilde{Z} \to Z$ be the projection map onto equivalence classes and let $\tilde{D}_+ \subset \tilde{Z}$ be given by
\[
\tilde{D}_+ := \left\{  (x,p) ~|~ f(x) > 0  \right\}.
\]
A measure $m_Z$ on $Z$ is said to be an admissible measure for the warped product whenever
\[
m_Z \ll P_\sharp \Big( \left( m_X \otimes m_Y  \right)\mathbf{1}_{\tilde{D}_+} \Big).
\]
\end{definition}
 In order to be able to reduce the transport problem between two measures to a problem on the base space, the two probability measures need to be spread out uniformly over the space. 
 \begin{definition}[spatially uniform product measures]
 	Let $Z$ be as in above. A measure $\mu$ on $Z$ is called spatially uniform whenever 
 	\[
 	\mu = P_\sharp \left( \tilde{\mu} \otimes m_Y  \right)
 	\]
 	for some measure $\tilde{\mu}$ on $X$.
 \end{definition}
\subsection{$c$-fine geodesic spaces}
\label{Sec:finespaces}
Since our proof takes advantage of the existence of optimal maps on the base space, we will need to assume some regularity assumptions on the base space $X$ strong enough to guarantee existence of optimal maps. These will be called fine base spaces. 
\begin{definition}[$c$-fine base space]
	Let  $\left(X, d_X, m_X\right)$ be a given geodesic metric measure space and $c(x,y)$ a cost function. In these notes, $X$ is said to be a $c$-fine geodesic space if the Monge problem can be solved for two absolutely continuous measures with finite transport cost between them i.e. if there always exist an optimal transport map between two absolutely continuous probability measures whose transport cost is finite. We will refer to $X$ as a fine base space if $X$ is $c$-fine for all cost functions that are convex functions of distance. 
\end{definition}
\subsubsection{A glossary of fine spaces} Suppose the cost function is of the form $c=  h\left(d(x,y)\right)$ and $h$ is strictly convex; for such cost functions, $c$-fine spaces include (with overlaps) Riemannian manifolds and Alexandrov spaces; see~\cite{McCann-polar,Bert}. Also based on~\cite{Cav-Hues}, proper non-branching metric measure spaces are $c$-fine provided that a mild non-contraction property (\cite[Assumption 1]{Cav-Hues}) holds and in particular, proper non-branching spaces which satisfy a measure contraction property (as in~\cite{Sturm-2006-II, Ohta-MCP}) are $c$-fine.
\par We know that in the special case when $h(x) = x^2$ (the squared distance cost), the class of $c$-fine spaces is richer and furthermore includes (with overlaps) $CD(K,N)$ ($1\le N \le \infty$) spaces, the Heisenberg group and essentially non-branching metric measure spaces satisfying $MCP(K,N)$ ($1\le N<\infty$); see~\cite{AR-04, AR-14, Gigli-optimalmaps, GRS, Cav-Mon}. 
\par When $h$ is not necessarily strictly convex, the solvability of Monge's problem is very delicate. For the cost given by a general norm in $\R^n$, the existence follows from the long line of works~\cite{Sudakov, Evans-Gangbo, CFM, Trudinger-Wang, AKP, Caravenna, CD-1, Cd-2}. In particular, it follows that $\R^n$ equipped with the Lebesgue measure is $c$-fine when $c= h\left( d(x,y) \right)$ is a norm. For complete Riemannian manifolds and when $c(x,y) = d(x,y)$, the existence of optimal maps for two absolutely continuous measures with compact supports follows from~\cite{Feldman-McCann}. 
\par Solvability of Monge's problem for the distance cost in non-branching geodesic metric measure spaces has been studied in~\cite{Bian-Cav} and in particular, a geodesic metric measure space which satisfies a measure contraction property $MCP(K,N)$ is $c$-fine where $c$ is the distance cost. 
\par The results of this section work for any $c$-fine base space however, we will only apply them to the case where $X$ is the interval $[-1,1]$ and $c=h\left( d(x,y)  \right)$ with $h$ convex so, $X$ is obviously $c$-fine. 
\begin{remark}
 Uniqueness of optimal maps in $c$-fine spaces is also provided in the majority of the cases mentioned above which we will not discuss here. We just point out that in a collapsed warped product space over a fine base space, optimal maps for transporting spatially uniform measures exist but they are, in general, not unique since the warping function becoming zero will cause a lot of geodesics to branch. 
 \end{remark}
\subsection{Transport of spatially uniform measures}
\begin{theorem}
Let $\left( X, d_X, m_X \right)$ be a $c$-fine base space for a cost function of the form $c(x,y) = h\left( d(x,y) \right)$ with $h$ non-negative and non-decreasing. Let $Z$ be as in above and $m_Z$ be an admissible background measure. Suppose $\mu_1$ and $\mu_2$ are two spatially uniform product measures of mass $1$ on $Z$, which are absolutely continuous with respect to $m_Z$, given by
$
\mu_1 = P_\sharp \left(  \tilde{\mu}_1 \otimes m_Y \right)
$
and 
$
\mu_2 = P_\sharp \left(  \tilde{\mu}_2 \otimes m_Y \right).
$
 Set
\[
\bar{c} \left( (x,p), (y,q)    \right) = h \left( d_X \left( x,y  \right)  \right)
\]
to be the cost function on $X$ corresponding to $c$. Then, the optimal total transportation costs satisfy
\[
\mathcal{T}_{\bar{c}} \left( \bar{\mu}_1, \bar{\mu}_2  \right) = \mathcal{T}_{c} \left( \mu_1, \mu_2  \right).
\]
where $\bar{\mu}_i$ is the pushforward probability measure, ${\mathrm{pr}_1}_\sharp \mu_i$, on $X$. Here, $\mathrm{pr}: Z \to X$ is the pushout of the projection map $\mathrm{Pr}: \tilde{Z} \to X$ by $P$ given by
\[
\mathrm{pr} \left( [x,p] \right) = x.
\]
\end{theorem}
\begin{proof}
Since $\mu_i (i=0,1)$ are absolutely continuous with respect to $m_Z$, we have $\tilde{\mu}_i \ll m_X$ ($i = 1,2 $). Set
	\[
	d \tilde{\mu}_i(x) = \tilde{\mu}_i(x) dm_X.
	\]
Then,
\[
1 = \int_Z \; d\mu_i = \int_{X \times Y} \;  \tilde{\mu}_i(x) \; dm_Xdm_Y = \vol(Y) \int_X \; d\tilde{\mu}_i, \quad i=1,2.
\]	
Let $\bar{\mu}_i$ be probability measures on $X$ given by the elements
\[
d \bar{\mu}_i = \vol(Y)\tilde{\mu}_i dm_X, \quad i=1,2.
\]
By definition, since $X$ is $c$-fine, there exists an optimal transport map, $T:X \to X$ between $\bar{\mu}_1$ and $\bar{\mu}_2$. Let $\bar{\pi}$ be the optimal plan for transporting $\bar{\mu}_1$ to $\bar{\mu}_2$ induced by $\bar{T}$ i.e. 
\[
\bar{\pi} = \left( \mathrm{id}_X \times \bar{T} \right)_\sharp \bar{\mu}_1.
\]
By absolute continuity of $\bar{\mu}_i$, it follows that $\bar{\pi}$ is absolutely continuous with respect to the product measure on $\tilde{Z}$ thus we write
\[
d \bar{\pi} = \bar{\pi} \; dm_X dm_X.
\]
Consider the map $\tilde{T}: \tilde{Z} \to \tilde{Z}$, given by
\[
\tilde{T} = \left( \bar{T}, \; \id_Y  \right)
\]
and set 
\[
\tilde{\pi} = \left(\id_{\tilde{Z}} \times \tilde{T} \right)_\sharp\left( \bar{\mu}_1 \otimes m_Y(Y)^{-1} m_Y \right), \quad \pi = (P, P)_\sharp \tilde{\pi}.
\]
By definition, the first marginal of $\tilde{\pi}$ is $\bar{\mu}_1 \otimes m_Y(Y)^{-1}m_Y$ and also for a product of measurable subsets $U \times V$ ($U$ and $V$ are measurable subsets of $X$ and $Y$ respectively), 
\begin{align}
\tilde{\pi} \left(  \tilde{Z} \times \left( U \times V \right)  \right) \notag &= \left( \bar{\mu}_1 \otimes m_Y(Y)^{-1}  m_Y \right) \left(  \tilde{Z} \times \left(  \bar{T}^{-1}(U) \times V \right)   \right) \notag  \\ &= m_Y(Y)^{-1} \bar{\mu}_2(U) m_Y(V) \notag \\ &= \tilde{\mu}_2\otimes m_Y \left( U \times V \right). \notag
\end{align}
Since 
\[
\tilde{\pi} \left( \tilde{Z} \times \left(U \times V\right)  \right)
\]
is multi-additive in $U$ and $V$ and the measures $m_X$ and $m_Y$ are $\sigma$-finite, by applying Carath\'{e}odory's extension theorem twice, we deduce there is a unique extension of $\tilde{\pi}$ to a measure on $\tilde{Z}$ and therefore for any measurable subset $A \subset \tilde{Z}$, we have
\[
\pi \left( \tilde{Z} \times A  \right) =  \tilde{\mu}_1\otimes m_Y \left( A \right).
\]
Since $\tilde{T}$ preserves the equivalence classes of $\sim$, it reduces to the map $T: Z \to Z$ given by
\[
T = P \circ  \left( \bar{T}, \; \id_Y  \right)  \circ P^{-1}.
\]
So,
\begin{align}
\label{EQN:Map}
\pi = \left( P, P  \right)_\sharp \tilde{\pi} &= P_\sharp \left(\id_{\tilde{Z}} \times \tilde{T} \right)_\sharp\left( \bar{\mu}_1 \otimes m_Y(Y)^{-1}m_Y \right) \notag \\ & = \left(  \id_Z \times \left(P \circ \bar{T} \circ P^{-1} \right)  \right)_\sharp P_\sharp \left( \bar{\mu}_1 \otimes m_Y(Y)^{-1}m_Y \right) \notag \\ & = \left( \id_Z \times  T  \right)_\sharp \mu_1.  \notag 
\end{align}
We claim $T$ is an optimal transport map between $\mu_1$ and $\mu_2$ and $\pi$ is an optimal plan. To do this, it is sufficient to only show $\pi$ is optimal. By construction, it is obvious that the first marginal of $\pi$ is $\mu_1$. Also, for $A \subset Z$, we have
\[
\pi \left( Z \times A \right) = \tilde{\pi} \left( \tilde{Z} \times P^{-1}(A) \right) = \tilde{\mu}_2\otimes m_Y \left( P^{-1}(A) \right) = \mu_2 \left( A \right).
\]
therefore, the second marginal is $\mu_2$. 
\par By construction of $\tilde{\pi}$, we know
\[
d\tilde{\pi}\left( (x,p), (y,q)   \right) = m_Y(Y)^{-1} \bar{\pi}(x,y)\; dm_Xdm_Y\; dm_Xdm_Y,
\]
and since all the cost functions are non-decreasing functions of distance, using the definition of the intrinsic distance on a warped product space and characterization of geodesics (see~\textsection\ref{Sec:Warped}), one gets
\[
c\left( [(x,p)], [(y,p)]   \right) = \tilde{c}\left( (x,p), (y,p) \right) = \bar{c}\left( x,y \right).
\]
where $\tilde{c}$ and $\bar{c}$ are respectively given by $h \circ \tilde{d}$ and $h \circ d_X$.
Therefore, total cost of the plan $\tilde{\pi}$ is
 \begin{align}
  \int_{\tilde{Z}^2}\; \tilde{c} \left( (x,p), (y,q)   \right) d\tilde{\pi} &= \int_{\tilde{Z}^2}\; \tilde{c} \left( (x,p), (y,q)   \right) d \left( \id_{\tilde{Z}} \times \tilde{T}  \right)_\sharp \left( \bar{\mu}_1 \otimes m_Y(Y)^{-1}m_Y   \right) \notag \\ &=  m_Y(Y)^{-1}\int_{\tilde{Z}}\; \tilde{c} \left( (x,p), \tilde{T}\left( x,p \right)  \right) \; d \bar{\mu}_1  d m_Y \notag \\ &=  m_Y(Y)^{-1} \int_{\tilde{Z}}\; \tilde{c} \left( (x,p), \left(  \bar{T}^{-1}(x), p   \right) \right) \; d \bar{\mu}_1  d m_Y  \notag \\ &=  m_Y(Y)^{-1} \int_{X} \int_Y \; \bar{c} \left( x, \bar{T}(x)   \right) d\bar{\mu}_1  \; d m_Y \notag \\ &= \int_{X} \; \bar{c} \left( x, \bar{T}(x)   \right) d\bar{\mu}_1 \notag
 \end{align}
which is the total cost of $\bar{\pi}$. Also, the total cost of the plan $\pi$ is
\begin{align}
\int_{Z^2}\; c \left( [(x,p)], [(y,q)]   \right) d\pi &= \int_{Z^2}\; c \left( [(x,p)], [(y,q)]   \right) d \left( id \times T  \right)_\sharp \mu_1 \notag \\ &= \int_{Z}\; c \left( [(x,p)], T\left( [(x,p)]   \right)  \right) d \mu_1 \notag \\ &= \int_{Z}\; c \left( [(x,p)], \left[ \left(\bar{T}(x), p \right) \right] \right) d P_\sharp \left( \tilde{\mu}_1  \otimes m_Y   \right) \notag  \\
&=  \int_{\tilde{Z}}\; \tilde{c} \left( (x,p), \left( \bar{T}(x), p \right) \right) d \left( \tilde{\mu}_1  \otimes m_Y   \right). \notag 
\end{align}
Therefore, the total costs of $\pi$, $\tilde{\pi}$ and $\bar{\pi}$ are the same. Now to prove optimality of $\pi$, let $\sigma$ be any other (absolutely continuous) admissible plan for the transport of $\mu_1$ to $\mu_2$ given by the element
\[
d\sigma = \sigma\left( [(x,p)], [(y,q)]   \right) dm_Z dm_Z.
\]
Let $\mathrm{pr}: Z \to X$ be the pushout of the projection map $\mathrm{Pr}: \tilde{Z} \to X$ by $P$ i.e.
\[
\mathrm{pr}\left( [(x,p)]  \right) = x
\]
and set $\bar{\sigma} = \left( \mathrm{pr}, \mathrm{pr}  \right)_\sharp \sigma$. The marginals of $\bar{\sigma}$ are $\bar{\mu}_1$ and $\bar{\mu}_2$ since for a measurable subset $U \subset X$,
\begin{align}
	\bar{\sigma} \left( U \times X  \right) &= \int_{\left(U \times X\right)}  \left( \mathrm{pr}, \mathrm{pr}  \right)_\sharp \sigma = \int_{P\left(U \times Y\right) \times P\left(X \times Y\right) } \; d\sigma  \notag \\ &= \mu_1 \left( P\left(U \times Y\right)  \right) = m_Y(Y) \tilde{\mu}_1 \left( U \right) = \bar{\mu}_1 (U)  \notag
\end{align}
and similarly $	\bar{\sigma} \left( X \times U  \right) = \bar{\mu}_2(U)$. So $\bar{\sigma}$ is an admissible plan for transporting $\bar{\mu}_1$ to $\bar{\mu}_2$. By optimality of $\bar{\pi}$, we have
\begin{align}
	\int_{Z^2} \; c\left(  [(x,p)], [(y,q)] \right) \; d\sigma &\ge \int_{Z^2} \;  \bar{c}(x,y) \;  \sigma\left( [x,p)], [(y,q)]   \right) dm_Z dm_Z \notag \\ &\ge \int_{X^2} \bar{c}(x,y) d\bar{\sigma}  \notag \\ &\ge \int_{X^2} \;  \bar{c}(x,y) \; d\bar{\pi} \notag \\ &= \int_{Z^2}\; c(x,y) \; d\pi \notag
\end{align}
hence, optimality of $\pi$. 
\end{proof}
\section{Set up of weak super Ricci flow}
\label{SEC:setup}
\subsection{Ricci flow pseudo-metric measure spacetime}
\label{subsec:RF-MMS-NP}
We will introduce the Ricci flow pseudo-metric measure spacetimes by essentially requiring the flow to stay within a (possibly singular) multiply twisted product regime. This is very reasonable framework especially for flows with symmetry; indeed for a rotationally symmetric spherical Ricci flow neckpinch continued by a possible rotationally symmetric smooth forward evolution, one expects any spacetime describing the flow to be of this form.
\begin{definition}[pseudo-metric measure spacetime of product form]\label{DEF:MMST}
	Let $\left(  \mathcal{X}, \mathbf{t}, d, m   \right)$ be a quadruple consisting of a complete metric measure space $\left(   \mathcal{X}, d, m    \right)$ and a surjective time function $\mathbf{t}: \mathcal{X} \to I$ in which $I$ is an interval. We say this quadruple is a \emph{pseudo-metric measure spacetime of the product form} if it isometrically splits as
	\[
	\left(   \mathcal{X}, d, m    \right) \equiv \left( I \times X', d_{Euc} \oplus d'_t,  L^1 \otimes m'_t \right)
	\]
	where $\left(  X', d'_t, m'_t  \right)$ is a time-varying pseudo-metric measure space with pseudo-distance function $d_{t'}$ and Borel measure $m'_t$. A Borel measure here is with reference to the $\sigma$-algebra generated by open balls with respect to pseudo-distance $d_{t'}$. We furthermore assume $d'_t(x,y)$ is continuous in $t$. We say this quadruple is an $(n+1)$-dimensional \emph{geometric spacetime of the product form} if $m'_t$ is the $n$-dimensional Hausdorff measure $\mathcal{H}^n_{d'_t}$. 
	\par The points of a metric measure spacetime of product form will be signified by pairs $(t,x')$. 
\end{definition}
\par From here onwards, a singular Riemannian metric $g_{sing}$ or pseudo-metric on a product manifold $M = B \times F_1 \times F_2 \times \dots \times F_r$ is referred to a \emph{non-negative} definite multiply twisted symmetric two tensor of the form
\[
g_{sing}= \phi^2 g_B \bigoplus_{i=1}^r \psi_i^2 g_{F_i}
\]
where the twisting functions $\phi$ and $\psi_i$ are all non-negative \emph{and apriori continuous and of bounded variation} functions defined on $M$ so that a pseudo-distance could be generated and $g_B$ and $g_{F_i}$ are respectively positive definite Riemannian metrics on the base manifold $B$ and on the fiber manifolds $F_i$. 
\par Let $\left(  \Omega, g   \right)$ be a Riemannian domain (open connected subset of some Riemannian manifold). In what follows, when the domain $\Omega$ is clear from the context, $d_g$ will denote the distance \emph{induced} by $g$ on $\Omega$.

\begin{definition}[regular flow points]
	Let $\left( t_0, x'  \right)$ be a spacetime point in a geometric spacetime of the product form. $\left( t_0, x'  \right)$ is said to be an $(n + 1)$-dimensional regular flow point if there exists radius $ r > 0$ and $\epsilon >0$  and a fixed subset $\mathcal{N} \subset  X'$ such that 
	\begin{enumerate}
	\item $\left( B^{d'_t}_{r}(x'),  d'_t\right)$ is a metric space i.e. $d'_t$ restricted to $B^{d'_t}_{r}(x')$ is a distance function.
	\item $\mathcal{N}$ is open in $\left(  X', d'_t  \right)$ for $\left| t - t_0 \right| < \epsilon$;
	\item $B^{d'_t}_{r}(x')$ is an $n$-dimensional open manifold containing $\mathcal{N}$ for $\left| t - t_0 \right| < \epsilon$;
	\item Restricted distance functions $d'_t\big|_{\mathcal{N}}$ are induced by Riemannian metrics $g'_t$ on $B^{d'_t}_{r}(x')$;
	\item The family $g'_t$ is $C^1$ in $t$. 
	\end{enumerate}
 The set of all $(n+1)$-dimensional regular flow points are denoted by $\mathcal{R}_{n+1}$. The top dimensional regular set is denoted by $\mathcal{R}$. For fixed $t\in I$, $\mathcal{R}_t$ denotes $t$-times slice of $\mathcal{R}$.  
\end{definition}
\begin{definition}[Ricci flow spacetime]\label{DEF:RFMMST}
	A (geometric) pseudo-metric measure spacetime of product form $\left(  \mathcal{X}, \mathbf{t}, d, m\right)$ is said to be a \emph{Ricci flow spacetime} whenever the following hold:
	\begin{enumerate}
		\item The corresponding product structure is of the form
		\[
		\left( I \times \underbrace{ B^n \times F^{m_1}_1 \times F^{m_2}_2 \times \dots  F^{m_r}_r}_{X'}, d_{Euc} \oplus d'_t,  L^1 \otimes \mathcal{H}^k_{d'_t}   \right)
		\]
		where $B$ and $F_i$ are compact manifolds possibly with boundaries and $k = n + \sum\limits_1^r m_i$.
		\item There exists a singular Riemannian metric $g_{sing}$ on $I \times X'$ of the form
		\[
		g_{sing} = dt^2 \oplus \phi^2 g_B \oplus  \psi_1^2 g_{F_1} \oplus \psi_2^2 g_{F_2} \oplus \dots \oplus \psi_r^2 g_{F_r}
		\]
		where $g_B$ and $g_{F_i}$ are non-degenerate Riemannian metrics. $\phi: I \times X' \to \R_{\ge 0}$ and $\psi_r: I \times X' \to \R_{\ge 0}$ are apriori only piecewise smooth in $x$ and $t$. 
		\item The pseudo-distance functions $d'_t$ are consistent with the product structure in the sense that for every fixed $t_0$, the distance $d'_{t_0}$ is locally generated by the singular Riemannian metric
		\[
		g_t = \phi^2 g_B \oplus  \psi_1^2 g_{F_1} \oplus \psi_2^2 g_{F_2} \oplus \dots \oplus \psi_r^2 g_{F_r}.
		\]
			\item On the regular set $\mathcal{R}$ consisting of top dimensional regular flow points, one has
		\[
		\frac{\partial g_t}{\partial t}  = - 2\Ric\left( g_t \right).
		\]
	\end{enumerate}
A point in a Ricci flow metric measure spacetime is denoted by $\left(t, x,p  \right)$ with $t \in I$, $x \in B$ and $p \in F_1 \times F_2 \times \dots \times F_r$. A \emph{Ricci flow spacetime} $\left(  \mathcal{X}, \mathbf{t}, d, m\right)$  is said to be a \emph{rotationally symmetric spherical spacetime} if  $B = [-1,1]$, $r=1$, $F_1 = \sphere^n$, $g_{F_1} = g_{\text{can}}$, and $\phi$ and $\psi_1$ are independent of $p$ (the twisted product is a warped product).
\end{definition}
\begin{definition}[Ricci flow spacetime through ``a'' pinch singularity]
We say a Ricci flow spacetime as in Definition~\ref{DEF:RFMMST} is a \emph{Ricci flow spacetime through a pinch singularity} if in addition,
	\begin{enumerate}
	\item There exists a singular time $T \in I$  such that the $\mathbf{t}^{-1}\left( (-\infty,T) \cap I \right)$ is a smooth manifold of dimension $1 + n+\sum_1^r m_i$, the time slices of which are smooth $n+\sum_1^r m_i$ dimensional manifolds with smooth Riemannian metrics $g_t$ (as in item 3 in Definition~\ref{DEF:RFMMST}). 
	\item There exists $x \in \mathring{B}$ and some $i\in \{1,2,\dots,r\}$ such that for all $t \ge T$ and all $p \in F$,  $\psi_i$ is zero at $\left(t,x, p \right)$.
	\item For all $i\in \{1,2,\dots,r\}$ and all $t\ge T$, the zero set of $\psi_i$ is non-increasing in $t$ (with respect to the subset order). This means no new singular point appears after time $T$ but singular points could possibly recover in time.
	\end{enumerate}
\end{definition}
Below, we give the definition of a rotationally symmetric spherical spacetime through a non-degenerate neckpinch singularity which are the spacetimes to which our main theorem applies.
\begin{definition}\label{DEF:RFSP-NP}
	Let $\left(  \mathcal{X}, \mathbf{t}, d, m\right)$  be a rotationally symmetric spherical metric measure spacetime through a pinch singularity. We say this quadruple is a {\bf rotationally symmetric spherical Ricci flow spacetime through ``a'' neckpinch singularity} if in addition
	\begin{enumerate}
		\item $\mathbf{t}^{-1}\left( (T, \infty) \cap I \right)$ is \emph{isometric} to the singular product spacetime
		\[
		\left(  \left((T, \infty)\cap I \right) \times [-1,+1] \times \sphere^n, g_{sing} = dt^2 \oplus \phi(t,x)dx^2 \oplus \psi(t,x)^2 g_{\sphere^n} \right),
		\]
		where for every $t \ge T$, the set $\psi > 0$ is the disjoint union of finitely many open (relative to $[-1,1]$) {\bf time-independent} subintervals $\left\{ I_\lambda\right\}_{\lambda \in \Lambda}$ of $[-1,+1]$. 
		
		\item For each $\lambda \in \Lambda$, and each $t>T$, the restricted Riemannian metric
	\[
	 \phi(t,x)dx^2 \oplus \psi(t,x)^2 g_{\sphere^n}, \quad x \in I_\lambda
	\]
		smoothly extends to a Riemannian metric, $\bar{g}^\lambda_t$ on the sphere $\sphere_{n+1}$ which depends smoothly on $t$ for $t>T$ and extends the Ricci flow equation i.e.
		\[
		\frac{\partial \bar{g}^\lambda_t}{\partial t}  = - 2\Ric\left( \bar{g}^\lambda_t \right), \quad \forall t>T.
		\]

	\item As $t \to T_+$ and on the regular set $\mathcal{R}$ consisting of top dimensional regular flow points, $g_t$ is a smooth forward evolution of Ricci flow. Namely, 
		\[
		 	\lim_{t \to T^+} g_t|_{\mathcal{O}} = g_T|_{\mathcal{O}}, \quad \text{in $C^\infty$-norm on open sets $\mathcal{O} $ with $\bar{\mathcal{O}} \subset \mathcal{R}_t$}.
		\]
notice that it follows from our hypotheses that $\mathcal{R}_t$ for $t \ge T$ does not depend on $t$.
	\end{enumerate}
\end{definition}
As we saw, a Ricci flow spacetime through a neckpinch singularity has a time-independent singular set. Some results hold under a weaker condition.
\begin{definition}[a non-recovering flow]\label{DEF:NONRECOV}
	Let $\left(  \mathcal{X}, \mathbf{t}, d, m\right)$ be \emph{Ricci flow spacetime} through a \emph{pinch singularity} (as in Definition \ref{DEF:RFMMST}). Define the time dependent $k$-th singular stratum
	\[
	\mathcal{S}_t^k := \left\{ x'=\left( x, y_1, \dots, y_r  \right) \quad | \quad \text{precisely $k$ of $\psi_i$'s vanish at $(t,x')$}  \right\}
	\]
	with the obvious iteration
	\[
	\mathcal{S}_t^1 \subseteq \mathcal{S}_t^2 \subseteq \dots \mathcal{S}_t^{r-1}\subseteq \mathcal{S}_t^r.   
	\]	
	The time dependent singular set is then
	\[
	\mathcal{S}_t := \mathcal{S}_t^1 \cup \mathcal{S}_t^2 \cup \dots \cup \mathcal{S}_t^{r}.
	\]
	We say the flow is weakly non-recovering if the time dependent singular set $\mathcal{S}_t $ is non-decreasing in $t$ namely if
	\[
	t_1 < t_2 \quad \Longrightarrow \quad \mathcal{S}_{t_1} \subseteq \mathcal{S}_{t_2},
	\]
	and strongly non-recovering whenever all $k$-th singular strata are non-decreasing in $t$. 
\end{definition}
It is immediate from Definitions~\ref{DEF:RFMMST} and \ref{DEF:NONRECOV} that a strongly non-recovering spacetime through pinch singularity has time-independent singular set for $t\ge T$. 
\subsection{The resulting time-dependent metric measure space}
There is a natural way to assign a time-dependent metric measure space to $\left(  \mathcal{X}, \mathbf{t}, d, m   \right)$. The time slices 
\[
\left(  X', d'_t, m_t   \right)
\]
are time-dependent pseudo-metric measure spaces and therefore the associated time-dependent metric measure space is
\[
\left( \frac{X'}{\sim_t},~{P_t}_\sharp d'_t,~{P_t}_\sharp m_t   \right).
\]
where $P$ is the quotient map corresponding to the equivalence relation
\[
x \sim_t y \iff d'_t(x,y) = 0.
\]
A priori, the underlying space in this time-dependent metric measure space also varies with time. 
\addtocontents{toc}{\setcounter{tocdepth}{1}}
\subsection*{Example}
If a rotationally symmetric spherical Ricci flow $\left( \sphere^{n+1}, g(t) \right)$ develops a non-degenrate  neckpinch singularity at the first singular time $T$ which satisfies the asymptotics given in~\cite[Table 1]{ACK}, then existence of a smooth forward evolution of Ricci flow for short time $\epsilon$ exists~\cite{ACK}. For a fixed such smooth forward evolution, ``an'' associated non-recovering rotationally symmetric spherical Ricci flow through a neckpinch singularity with time interval $[0, T+\epsilon]$ is then comprised of the following time slices
\begin{itemize}
	\item For $t<T$, time slices are just smooth Riemannian spaces$\left( \sphere^{n+1}, g(t) \right)$;
	\item For $t=T$, the time slice is the singular Riemannian space $\left( \sphere^{n+1}, g(T) \right)$;
	\item For $T < t < T+\epsilon$, the time slices are union of the smooth Riemannian spaces obtained from the smooth forward evolution and the singular set with time-dependent length given by
	\[
	g(t)_{\text{sing}}(x,t) = \phi_{\text{sing}}(t,x)dt.
	\]	
	Notice that our definition allows $\phi_{\text{sing}}$ to also evolve in time which leads to the change in length of the neck but the singular set itself stays unchanged. 
\end{itemize}
\addtocontents{toc}{\setcounter{tocdepth}{2}}
\subsection{Diffusions in time-dependent metric measure spaces}
\label{Sec:diffusions}
We wish to define weak super Ricci flows by the coupled contraction property. For this, we will need a theory of diffusions i.e. theory of solutions to conjugate heat equation in a time-dependent metric measure space. To this end, we will consider the time dependent Cheeger-Dirichlet energy functionals. There are some subtleties in applying the results of~\cite{Kopfer-Sturm} here. One is that we need our measures not to be with full topological supports since we want to allow for interval neckpinches, the other is whether the time-dependent Cheeger-Dirichlet energy functionals well-defined and then is the question of the conditions that need to hold so that the dynamic conjugate heat equation is well-posed. 
\par In what follows $\mathcal{E}_\tau$ is the Cheeger-Dirichlet energy functional on the $\tau$-time slice. The Cheeger-Dirichlet energy of $f \in L^2$ is simply the $L^2$-norm of the minimal upper gradient which coincides with pointwise Lipschitz constant of $f$ when $f$ is Lipschitz. 
\par In the setting of spherical Ricci flow spacetimes, we need to consider two prototype pictures for spherical Ricci flow neckpinch through singularity: interval pinching and single point pinching. In both cases the time slices do not satisfy weak Ricci curvature bounds in the sense of Lott-Sturm-Villani however, they satisfy measure contraction properties (see \cite{Sturm-HK}) with exceptional sets so one can define a time dependent Dirichlet form as in~\cite{Sturm-HK} which is local and its energy density is well-defined on the smooth parts coinciding with the Cheeger-Dirichlet energy density. Also since the time slices are infinitesimally Hilbertian, one can directly use the definition of the Cheeger-Dirichlet energy functional to see that for any $f \in W^{1,2}$ where $W^{1,2}$ denotes the domain of the Cheeger-Dirichlet form, the minimal weak upper gradient of $f$ equals $0$ outside the (time-independent) support of the time dependent measures and equals the norm of the distributional derivative on the support (which is a smooth manifold). Now, consider the framework set up in~\textsection\ref{SEC:Framework}. The following can be easily verified.
\begin{itemize}
	\item Distances in time slices satisfy the log-Lipschitz condition on time intervals of the forms $[0,T-\epsilon]$ and $[T+\epsilon, T']$ for any fixed $\epsilon > 0$ but the log-Lipschitz condition breaks down at the singular time. 
	
	\item Restricting the time-dependent metric measure space to the time interval $[0,T-\epsilon]$, we have a smooth Ricci flow with bounded curvature so obviously the conditions in \cite{Kopfer-Sturm} hold and also the unique solutions to dynamic conjugate heat equation are just the classical solutions. 
	
	\item Restricting the time-dependent metric measure space to the time interval $[T+\epsilon, T']$ and setting $\tau = T' - \tau$. Then for $0 \le \tau \le T' - T + \epsilon$, the distance on time slices satisfy the log-Lipschitz property and the logarithmic density $f_\tau$, is smooth on the smooth parts, can be chosen arbitrary outside the support and in case of one point pinching, it is discontinuous at the double point.
	
	\item In the case of interval pinching, the logarithmic densities $f_\tau$ (as functions in $W^{1,2}$) have versions that are smooth on the support and Lipschitz continuous on the whole time slice.

	\item In the case of single point pinching, the logarithmic densities $f_\tau$ (as functions in $W^{1,2}$) have versions that are smooth away from the double point and possibly discontinuous at the double point.
	
	\item The minimal weak upper gradient of $f \in L^2$ restricted to $M_1$ and $M_2$ (resp.) coincides with the minimal weak upper gradient of $f|_{M_1}$ in $M_1$ and of $f|_{M_2}$ in $M_2$ (resp.)
\end{itemize}
\par These properties imply that on the time intervals of the form $[0,T-\epsilon]$ and $[T+\epsilon, T']$, the flow satisfies the conditions on~\cite[Page 17]{Kopfer-Sturm} except for the fact that in the interval pinching case, on the time interval $[T+\epsilon, T']$, the measures do not have full topological support so solutions of the (weakened) conjugate heat equation \cite[Definition 2.4]{Kopfer-Sturm} can be prescribed arbitrarily on the neck interval. This means, starting from an initial $L^2$ funciton, diffusions exist but are not unique. In order to get a well-posed theory of diffusions, we will need to make an adjustment to the domain function space. These observations lead us to consider extended diffusion spaces. 
\subsubsection{Diffusions on extended time-dependent metric measure spaces}
\label{Sec:extended-dif}
Let $\left( X, d_\tau, m_\tau \right)$ be  a {\bf compact} time-dependent geodesic metric measure space. Suppose $d_\tau$ satisfies the log-Lipschitz condition (so the topologies do not change along the flow) and $m_\tau$ are mutually absolutely continuous which are not necessarily with full topological support. Suppose $\supp\left(  m_\tau  \right)$ is time-independent with {\bf time independent connected components} that are totally geodesic in $\left( X, d_\tau \right)$ for all $\tau$. These connected components are closed hence compact and there is a positive lower bound on distances between these connected components with respect to all $d_\tau$. The latter follows from the log-Lipschitz assumption. Let 
\[
L^2\left( \supp\left(  m_\tau  \right), m_\tau \right)  \hookrightarrow L^2\left( X, m_\tau   \right)
\]
be the compact Hilbert space embedding defined via extension of functions by zero. Denote the image by $\mathcal{H}$. Set $m_\diamond := m_0$ and let $\mathcal{E}_{\diamond}$ and $\Gamma_\diamond$ to be fixed background strongly local Dirichlet form and its associated square field operator. Set $\mathcal{F} := \mathit{Dom}\left( \mathcal{E}_\diamond \right) \cap \mathcal{H}$. In general $\mathcal{F}$ is only Banach. Now we can construct the function space $\mathcal{F}_{\tau_1, \tau_2}$ just as in \cite{Kopfer-Sturm}. Assume $m_\tau = e^{-f_\tau}m_\diamond$ with regularity assumption
\[
\Gamma_\diamond \left(  f_\tau \right) \le C,\quad \text{and} \quad  \forall x \in \supp\left( m_\diamond  \right),\quad \forall t.
\]
where uniform ellipticity condition on $\mathcal{E}_t$  and diffusion conditions on the square field operators $\Gamma_t$ hold (see \cite[Page 17]{Kopfer-Sturm}) hold for all $t$ and for all $x$ in the time-independent support of $m_\tau$. Then automatically the required conditions hold on each connected component of the support thus the heat and conjugate heat equations formulated as in \cite[Definitions 2.2 and 2.4]{Kopfer-Sturm}, are well-posed on $L^2 \left(  \mathcal{C}, m_\tau|_{\mathcal{C}}  \right)$. This in turn implies the well-posedness of dynamic heat and conjugate heat equation in the function space $\mathcal{F}_{\tau_1, \tau_2}$ in time-dependent metric measure space $\left( X, d_\tau, m_\tau \right)$. 
\par Such a time-dependent metric measure space equipped with its time-dependent Cheeger-Dirichlet forms and the function spaces as described above gives rise to a well-posed diffusion theory as discussed above and these diffusions will be called extended diffusions. 
\par The regularity of diffusions is determined by the regularity of the dynamic heat kernel via the duality relation~\cite[Theorem 2.5]{Kopfer-Sturm}. Of most important regularity type to us is the continuity of the dynamic heat kernel on the support since this is what differentiates interval pinching from single point pinching. As we will explain in below continuity issue is also hand in hand with irreducibility of the Cheeger-Dirichlet energy functional and validity of a Poincar\'e inequality on the support. 

\par The existence of dynamic heat kernel can be found in e.g.~\cite{Lions-Magenes, RR-PDE, Pazy} (see also~\cite{Lierl20}) and parabolic Harnack inequalities (which provide us the joint continuity of the dynamic heat kernel) can be found in~\cite{Sturm-DirIII, BGT}. The ingredients for the parabolic Harnack inequality are a doubling condition and a Poincar\'e type inequality. In the single point pinching of Ricci flow, the Poincar\'e inequality will not hold at the double point hence we can not deduce the continuity of the dynamic heat kernel in the time-independent support of time-dependent measures. Indeed, since the time dependent Cheeger-Dirichlet forms are reducible, a simple super position argument as in Lemma~\ref{LEM:keylemma} shows the kernel has jump discontinuity on the support. For more on deriving Parabolic Harnack inequality and properties of reducibile Dirichlet forms, we will refer the reader to~\cite{Sturm-DirI, Sturm-DirIII}. 
\subsubsection{Two types of extended diffusion spaces}
\begin{definition}[normal and rough extended diffusion spaces]
	Let $\left( X, d_\tau, m_\tau \right)$ be  a time-dependent compact geodesic metric measure space. Suppose $d_\tau$ satisfies the log-Lipschitz condition and $m_\tau$ are mutually absolutely continuous with time-independent totally geodesic (with respect to all $d_\tau$) connected components. Take the time zero data to be our background data. Assume the logarithmic densities with respect to $m_0$ (and therefore, w.r.t. any $m_\tau$) are uniformly Lipschitz in time and satisfy
	\[
	\Gamma_\diamond \left(  f_\tau \right) \le C, \quad \forall t,\quad \text{and} \quad  \forall x \in \supp\left( m_\diamond  \right), \quad \forall t.
	\]
	Suppose the Cheeger-Dirichlet energies are all defined for all times and satisfy the uniform ellipticity and diffusion properties on the time-independent support; see~\textsection\ref{Sec:extended-dif} and~\cite[Page 17]{Kopfer-Sturm}. Then the heat and conjugate heat equations as formulated in~\cite[Definitions 2.2 and 2.4]{Kopfer-Sturm} are well-posed in $\mathcal{H}$. We will refer to solutions of the conjugate heat equation in $\mathcal{H}$ as extended diffusions and to such a time-dependent metric measure space as an extended diffusion space. However for simplicity, we will drop the adjective ``extended'' from here on. 
	\par We will be concerned with two types of diffusion spaces:
\begin{itemize}
			\item A diffusion space in which the dynamic heat transition kernel is continuous in time and in each space variables, on the time-independent support, is called a {\bf normal diffusion space}. This is asking for extra so a normal diffusion space will be more regular with more regular diffusions on it. In particular, when the topology induced by the intrinsic metric on the time-independent support coincides with the time-independent topology induced by the time-dependent distance functions and when Poincar\'e inequality and a doubling condition holds, we get a continuous dynamic heat kernel~\cite{Sturm-DirIII}; see also~\cite{Lierl-18}.
		
			\item A diffusion space as described without further requirements on the regularity of the dynamic heat kernel is called a {\bf rough diffusion space}; namely, we have a rough diffusion space when the dynamic heat transition kernel $p\left(\tau_1,x,\tau_2,y\right)$ is apriori only in $L^2$ in terms of each space variable (with times and other space variable kept fixed). The diffusion spaces described in~\textsection\ref{Sec:extended-dif} all satisfy this property and hence are rough diffusion spaces; see~\cite[Proposition 6.3]{Lierl-18}. 
\end{itemize}
As we will see, existence of a single point pinch will cause the heat kernel to fail to be continuous in each space variable. 
\end{definition}
\subsubsection{Characterizing diffusion components in our setting}
\begin{lemma}\label{LEM:keylemma}
	With the notations of~\textsection\ref{SEC:Framework} and definitions of~\textsection\ref{Sec:extended-dif}, If $u_1(x,\tau)$ and $u_2(x,\tau)$ are smooth diffusions (resp. dynamic heat solutions) on $M_1$ and $M_2$, then 
	\[
	u(x,\tau) = u_1(x,\tau) \mathbf{1}_{M_1}(x) + u_2(x,\tau) \mathbf{1}_{M_2}(x)
	\]
	is in the unique (in $L^2\left( X \right)$) diffusion (resp. dynamic heat solution) with initial data
	\[
	u(x,0) = u_1(x,0) \mathbf{1}_{M_1}(x) + u_2(x,0) \mathbf{1}_{M_2}(x).
	\]
	To wit, the diffusions and dynamic heat solutions both evolve on $M_1$ and $M_2$ independently. In the interval pinching case, the heat transition kernel is continuous in time and space variables on the time-independent support but in the single point pinching case, the heat transition kernel is not continuous in space variables on the support (the continuity exactly breaks down at the double where $M_1$ and $M_2$ are joined).
Consequently, in the case of a single point pinching, the heat kernel is not continuous in space variables. 
\end{lemma}
\begin{proof}
In our setting, the Cheeger-Dirichlet energy $\mathcal{E}_t$ localizes to the smooth pieces $M_1$ and $M_2$ by which we mean
\[
\mathcal{E}_t (f) = \int_{M_1} \left| D \left( f |_{M_1}  \right) \right|_w \; dvol_{g_1(\tau)} +  \int_{M_2} \left| D \left( f |_{M_2}  \right) \right|_w \; dvol_{g_2(\tau)},
\]
where $|D \cdot|_w$ denotes the minimal upper gradient. Notice that since the underlying spaces are piecewise smooth, minimal upper gradients do not depend on the $L^p$ norm we use to minimize. So clearly $u$ satisfies the defining equations for conjugate heat solutions as in \cite[Definition 2.4]{Kopfer-Sturm} (resp. dynamic heat solutions as in \cite[Definition 2.2]{Kopfer-Sturm}). By uniqueness, this is the unique (in $L^2$) solution. 
\par To show the dynamic heat transition kernel is not continuous at the double point. Let $y \neq p$ be a point in $M_1$ and $x \neq p$ a point in $M_2$. Then for $\tau_1 < \tau_2$ and according to the part one in this lemma and since smooth dynamic heat kernel exists and is positive (see~\cite{Guenther02}), we have
\[
h\left( \tau_2,y,\tau_1,x \right) = 0, \quad h\left( \tau_2,y,\tau_1,p \right) \neq 0 , \quad h\left( \tau_2,x,\tau_1,p \right) \neq 0
\]
which upon letting $x$ or $y$ converge to $p$ will imply the heat transition kernel is not continuous at the double point $p$ (meaning, admits no continuous representative). 
\end{proof}
From here onwards, all spaces are assumed to be compact with no mention of topology (which might change in time) when the time-dependent topology is clear from context. 
\subsection{Weak Super Ricci flows associated to a cost $c$}
\label{Sec:WSRF}
By now we have the sufficient machinery to define weak super Ricci flows associated to a time dependent cost function $c_\tau$. The definition itself does not need any conditions on the cost function.
\begin{definition}[maximal diffusion components]
	Let $\left( X_\tau, d_\tau, m_\tau \right)$ be a time-dependent geodesic metric measure space. Suppose there is an interval $\left[  \tau_1, \tau_2 \right]$ on which $X_\tau = X$ and suppose $\mathcal{C} \subset X$ is totally geodesic in $\left( X, d_\tau \right)$ for all $\tau \in \left[ \tau_1, \tau_2 \right]$. We say $\mathcal{C}$ is a normal (resp. rough) diffusion component on time interval $\left[ \tau_1, \tau_2  \right]$ if the time-dependent metric measure space $\left( \mathcal{C}, d_\tau|_\mathcal{C}, m_\tau|_\mathcal{C}  \right)$, $\tau \in \left[ \tau_1, \tau_2  \right]$ is a normal (resp. rough) diffusion space. A maximal such diffusion component (with respect to inclusion along with restriction of distance and measure) will be called a maximal normal (resp. rough) diffusion component.
\end{definition}
We note that a maximal diffusion component on $\left[  \tau_1, \tau_2 \right] $ is maximal in space yet it could possibly be extended in time so the a maximal diffusion component is not necessarily maximal in time span.
\begin{proposition}
	Let $\left( \mathcal{X}, \mathbf{t}, d, m \right)$ be a spherical rotationally symmetric metric measure spacetime through a neckpinch singularity with singular time $T$. Suppose all the singularities at time $T$ are interval neckpinches. Then for any $T + \epsilon < T'$, $\mathbf{t}^{-1}\left( \left[ T+\epsilon, T'  \right]  \right)$ is a maximal normal diffusion component.
\end{proposition}
\begin{proof}
	It suffices to show the conjugate heat equation is well-posed in $\mathbf{t}^{-1}\left( \left[ T+\epsilon, T'  \right]  \right)$ in the extended function space and  the  dynamic heat kernel is continuous on the time-independent support. Set $\tau = T' - \epsilon$ then the resulting time-dependent metric measure space is 
	\[
		\left( \frac{\sphere^{n+1}}{\sim_\tau},~d_{g(\tau)},~\mathcal{H}^{n}_{d_{g(\tau)}}  \right).
	\]
This is easy to see since the time-independent support of the time-dependent metric measure space is comprised of two smooth pieces on which the time-dependent Cheeger-Dirichlet energy is well-defined and vanishes outside the support. The logarithmic densities are smooth, hence the dynamic conjugate heat equation is well-posed. Indeed, any diffusion with initial data $f(x) \in L^2\left( \supp(m_\tau) \right)$ is comprised of two diffusions evolving on separate smooth pieces. The dynamic heat transition kernel is unique so it must coincide with the smooth dynamic heat kernel on the smooth pieces. 
\end{proof}
\begin{proposition}
	Let $\left( \mathcal{X}, \mathbf{t}, d, m \right)$ be a spherical rotationally symmetric metric measure spacetime through a neckpinch singularity with singular time $T$. Then for any $ T + \epsilon < T'$, $\mathbf{t}^{-1}\left( \left[ T+\epsilon, T'  \right]  \right)$ is a maximal rough diffusion component but is not a maximal normal diffusion component. 
	\par Indeed, in presence of point pinching, a maximal normal diffusion component is a maximal totally geodesic component which does not contain the double points as its interior point. 
\end{proposition}
\begin{proof}
	It suffices to show $\mathbf{t}^{-1}\left( \left[ T+\epsilon, T'  \right]  \right)$ is a maximal diffusion component i.e. the conjugate heat equation is well-posed and to show that dynamic heat kernel does not admit a continuous (in each space variable) representative. Since the time slices are comprised of smooth pieces joined with intervals or double points. The logarithmic densities are smooth on the smooth pieces yet they are not smooth at the double points. Nevertheless, Cheeger-Dirichlet energy is well-defined and based on Lemma~\ref{LEM:keylemma}, the dynamic transtition heat kernel coincides with the smooth dynamic heat transition kernels on the smooth pieces. Again from Lemma~\ref{LEM:keylemma}, it follows that a maximal normal diffusion component can not contain a double point in its interior. 
\end{proof}
\begin{definition}[diffusion atlas, exceptional set]
	Let $\left(  \mathcal{X}, \mathbf{t}, d, m   \right)$ be a {\bf pseudo-metric measure spacetime} of product form with its corresponding {\bf time-dependent metric measure space} 
	\[
	\left( X'_t, d'_t, m'_t\right).
	\]
	We say a spacetime point $\left( t,x'  \right)$ in $\mathcal{X}$ is swept by normal (resp. rough) diffusions whenever there exists $T$ so that setting $\tau = T-t$, one can find a maximal normal (resp. rough) diffusion component $\mathcal{C}$ (in the corresponding time-dependent metric measure space) on a time interval $\left[  \tau_1, \tau_2 \right]$ where $x'$ is in the interior of $C$ and $\tau$ is in the interior of $\left[  \tau_1, \tau_2 \right]$. 
	\par The exceptional set of the spacetime is the complement of the points swept by diffusions.
	 \par A collection of maximal normal (resp. rough) diffusion components $\left( \mathcal{C}, d_\tau|_\mathcal{C}, m_\tau|_\mathcal{C}  \right)$ sweeping $m$-a.e. point in the spacetime will be called a rough (resp. normal) diffusion atlas for the spacetime. 
\end{definition} 
\begin{definition}[weak (refined) super Ricci flow]
	\label{DEF:WSRF}
Let $\left( X_\tau, d_\tau, m_\tau \right)$ be an extended time-dependent geodesic metric measure space and $c_\tau$ a time dependent cost function. We say this is a {\bf weak (refined) super Ricci flow associated to time-dependent cost $c_\tau$} (or a (refined) $c$-$\mathcal{WSRF}$ for short) if there exists {\bf a} (rough) diffusion atlas on it such that for any two diffusions $\mu_1(\tau)$ and $\mu_2(\tau)$ supported in any maximal diffusion component $\mathcal{C}_{max}$ in the said atlas, the time-dependent optimal total cost
		\[
		\tau \mapsto \mathcal{T}^{~\mathcal{C}_{max}}_{c_\tau}(\mu_1(\tau), \mu_2(\tau))
		\]
		is non-increasing in $\tau$, on any interval where the right hand side is defined and is finite. 
\end{definition}
\begin{remark}
	Requiring a normal diffusion atlas, in general, means we are using smaller and more regular diffusion components on which it is easier for dynamic coupled contraction to hold on these components which in turn allows for less regular spacetimes. Notice that our definition of weak super Ricci flow asks for an atlas with the coupled contraction property and not assuming coupled contraction property for all diffusion atlases. In this sense, our definition is a very weak definition.   
\end{remark}
\begin{definition}[weak (refined) super Ricci flow spacetime]\label{DEF:WeakRicciFlow}
	A {\bf pseudo-metric measure spacetime of product form} $\left(  \mathcal{X}, \mathbf{t}, d, m   \right)$ is said to be a {\bf weak (refined) super Ricci flow spacetime associated to a cost function $c$} ($c$-$\mathcal{WSRF}$ spacetime for short) if the corresponding time-dependent metric measure space
	\[
	\left( X'_t, d'_t, m'_t    \right),
	\]
is a weak (refined) supper Ricci flow. 
\end{definition}
\addtocontents{toc}{\setcounter{tocdepth}{1}}
\subsection*{Examples}
	Any smooth Ricci flow gives rise to a weak (refined) super Ricci flow spacetime associated to any cost function $c_\tau = h \circ d_\tau$ where $h$ is convex, up until the first singular time. A time-dependent metric measure space time with mutually absolutely continuous Borel measures which is a super Ricci flow in the sense of Sturm~\cite{Sturm-SRF} (satisfies dynamic convexity of Boltzmann entropy) and which satisfies the regularity assumptions in Kopfer-Sturm~\cite{Kopfer-Sturm} produces a weak (refined) super Ricci flow spacetime associated to the squared distance cost. A spherical Ricci flow neckpinch starting from $g_0 \in \AK_0$ and continued by smooth forward evolution after the singular time $T$ is a weak super Ricci flow but {\bf is not} a weak refined super Ricci flow.
\addtocontents{toc}{\setcounter{tocdepth}{2}}
\section{Proof of main results}
\label{SEC:PROOFofTHM1.1}
As we saw in~\textsection\ref{SEC:OTRot}, the optimal total cost of transporting spatially uniform probability measures in a rotationally symmetric setting equates the optimal total cost of an optimal transport problem on $\R$ for which we have a rich theory. In this section we show how this fact can be used to prove our main theorem. 
\subsection{A closer look at our framework} 
\label{SEC:Framework}
To begin, we describe the framework for our argument as was set up in~\textsection\ref{SEC:setup} and specialize it to spherical Ricci flow through non-degenerate neckpinches. Suppose $\sphere^{n+1}$ is equipped with an initial metric 
\[
g_0 \in \mathcal{ADM} \supseteq \tilde{\AK} \supset{\AK} \supset  \AK_0
\]
, i.e., $g_0$ is an initial metric for which the flow developes a non-degenerate neckpinch singulrity and which can be continued beyond the singular time by a smooth forward evolution. In particular if $g_0 \in \AK$, it follows from~\cite{AK1} that a Type 1 neckpinch singularity develops through the Ricci flow at some finite time $T < \infty$. Ultimately, we aim to prove that the singularity which develops occurs at single points assuming the resulting Ricci flow flow spacetime through the neckpinch is a weak super Ricci flow associated to a/ any cost function of the form $c_\tau = h \circ d_\tau$. This means there exists a normal diffusion atlas (with continuous dynamic heat kernel). If the Ricci flow spacetime is a weak refined super Ricci flow assocated to $c_\tau$, we show the flow has to be smooth up until the final time slice. 
\par By way of contradiction, we assume instead that there is rotationally symmetric spherical Ricci flow spacetime through neckpinch singularities. For $t \in [0,T)$, $g(t)$ is a smooth metric on $(-1,1) \times \sphere^{n}$. At the singular time $t = T$, we view $(\sphere^{n+1}, g(T))$ as a metric space equipped with the singular Riemannian metric $g(T)$; or rather, possibly finitely many singular manifolds joined by line segments or joined at single point pinches (double points) where the singular metric is obtained from the neckpinch degenertation; see Figures~\ref{fig:intervalpinching} and~\ref{fig:onepointpinching}. We will only consider the case where the neckpinch happens on one interval. The general case can be dealt with similarly.  So we assume the neckpinch occurs on \textbf{an} interval of positive length in the interior of $[-1,1]$ i.e.
\[
\lim_{t \nearrow T} |\!\Rm\!| (x,t) = \infty, \quad  \forall x \in (x_1, x_2) \times \sphere^n \quad  \text{ where } \quad  -1 < x_1 < x_2 <1,
\]
This in particular means we are assuming the Ricci flow does not develop singularities at the poles corresponding to $x = \pm 1$. This assumption is a fact when the initial metric has positive curvature at poles; see \cite{AK1}. Being a neckpinch means any blowup limit of Ricci flow at any singular point is a shrinking cylinder soliton. However we will not use this so the arguments will apply to singularities of possibly other blow up limits which from a purely Ricci flow point of view, we do not expect to exist in the rotationally symmetric setting due to the assumptions we have made. We also assume the spacetime is a weak super Ricci flow associated to a cost function which is a convex function of distance. We aim to show these assumptions lead to a contradiction; see Figures~\ref{fig:intervalpinching} and~\ref{fig:onepointpinching} for the shape of time slices. 
 \begin{figure}
	\centering
	\includegraphics[width=3.2in, height=4.5in]{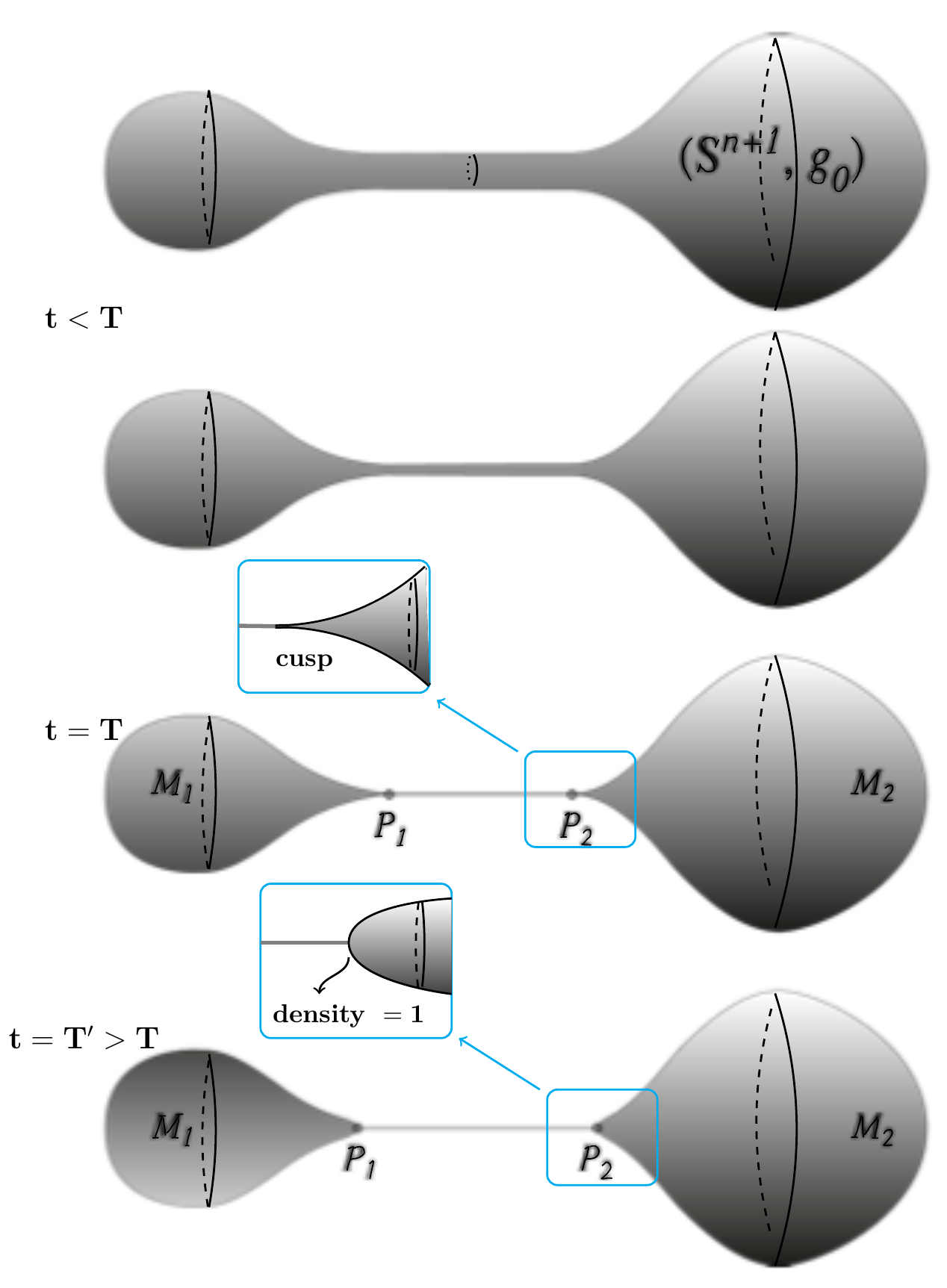}\\
	\caption{Flow through interval pinching is not a weak super Ricci flow associated to a convex cost.}
		\label{fig:intervalpinching}
\end{figure}
\par Let $M_1$ denote $\sphere^{n+1}$ with the singular metric $g_1$ arising through the Ricci flow as the limit of the Rimannian metrics $\left. g(t) \right|_{(-1, x_1) \times \sphere^n}$ as $t \nearrow T$. Similarly, let $M_2$ denote $\sphere^{n+1}$ with the singular metric $g_2$ which arises as the limits of metrics of $\left. g(t) \right|_{(x_2, 1) \times \sphere^n}$ as $t \nearrow T$. Label the points corresponding to $x_1$ and $x_2$, which are now the new poles of the degenerate spheres $M_1$ and $M_2$ by $P_1$ and $P_2$, respectively. 
\par We can write the metric space $(\sphere^{n+1}, g(T))$ as the union
\[
X :=\left( M_1 \sqcup M_2\right) \cup_h [0,1]
\]
which is obtained by first taking the disjoint union $\left( M_1 \sqcup M_2\right)$ and then identifying the boundary of $[0,1]$ to points in $\MM$ via a map $h: \{\{0\}, \{1\}\} \to M_1 \sqcup M_2$ where
\begin{align*}
	h(0) &= P_1 \in M_1 \\
	h(1) &= P_2\in  M_2.
\end{align*} 
We assume the interval joining $P_1 \in (M_1,g_1)$ to $P_2 \in (M_2, g_2)$ has length $L>0$, therefore the distance $d_X$ on $X$ is obtained from distances generated by $\left.g(T)\right|_{M_1}$ on $M_1$, $\left.g(T)\right|_{M_2}$ on $M_2$, and the one dimensional $L(x)d\lambda(x)$ on $[0,1]$ $d\lambda$ is the usual is the usual one-dimensional Lebesgue measure. Letting $d_g$ denote the distance metric induced by a smooth Riemannian metric $g$ and equipping $X$ with the distance $d_X$, it then follows that
\begin{equation}
\label{EQN:GHconvergenceBelow}
	\qquad \qquad \lim_{t \to T^-} d_{\text{GH}}\left((\sphere^{n+1}, d_{g(t)}),  (X,d_X)\right) =0.
\end{equation}
If we consider the Riemannian volume measures for $g(t)$ with $t<T$ and equip $X$ with the $(n+1)$-dimensional Hausdorff measure induced by $d_X$, the convergence is also true in measured Gromov-Hausdorff distance and in Intrinsic flat distance; see~\cite{lakzian-RF} for details. 
\par By hypothesis there exist at least one smooth forward evolution for this flow. Notice that the forward-time evolution of $M_1$ and $M_2$ is guaranteed and prescribed as in \cite[Theorem 1]{ACK} in the case where the initial metric is picked from the set $\AK_0$. Taking $g_i \; ( i=1,2)$ to be the initial singular metrics on $M_i$ obtained by restricting the singular metric $g(T)$, there exist complete, smooth forward evolutions $(M_i, g_i(t))$ for $t \in (T, T_i)$. We set  $T' := \min\{T+T_1, T+T_2\}$.
\par By definition of smooth forward evolution, the metrics $g_i(t)$ on $M_i$ (for $i = 1,2$) are smooth and converge smoothly to $g(T)|_{M_i}$ on precompact open sets in $M_i \setminus \{p_i\}$. Therefore, we have a one-parameter family of pseudo metrics $g(t)$ on $\sphere^{n+1}$ defined for $t \in [0, T')$ which are, in fact, smooth non-degenrate Riemannian metrics on all of $\sphere^{n+1}$ for $t \in [0, T)$ as well as on the open sets of $M_1$ and $M_2$ for $t \in (T, T')$. It is easy to check that this family of pseudo-metrics $g(t)$ induces the family of distance metrics $d_X(t)$ on the time slices of the correspinding time-dependent metric measure space via
\begin{equation}
\label{DEF:D metric}
d_X(t)(x,y) :=
\begin{cases}
d_{g_1(t)}(x,y)		& \text{if } (x,y) \in M_1 \times M_1,\\
d_{g_2(t)} (x,y)			& \text{if } (x,y) \in M_2 \times M_2,\\
d_{I,t} (x,y) := \left| \int_x^y\; L(t,u) du	\right|		& \text{if } (x,y) \in [0,1] \times [0,1],\\
d_{g_1(t)}(x, P_1) + d_{I,t} (0,1) + d_{g_2(t)}(y, P_2) 	& \text{if } (x,y) \in M_1 \times M_2,\\	
d_{g_i(t)}(x,P_i) +  d_{I,t} \left( p_i, y \right)		& \text{if } (x,y) \in M_i \times [0,1], ~i = 1,2.
\end{cases}
\end{equation}
It follows that
\begin{equation}
\label{EQN:GHconvergenceAbove}
\qquad \qquad \lim_{t \to T^+} d_{\text{GH}}\left( (\sphere^{n+1}, d_{g(T)}), (X,d_X(t)) \right) = 0.
\end{equation}
and again the convergence is even in the measured Gromov-Hausdorff and intrinsic flat sense; see~\cite{lakzian-RF}. 
\par Now assuming the rotationally symmetric spherical Ricci flow spacetime through singularity is also a weak super Ricci flo associated to a cost $c_\tau$, we can prescribe how the length of the interval joining $M_1$ and $M_2$ must evolve in a way that the flow satisfies our definition weak super Ricci flow; see Definition~\ref{DEF:WSRF}.
\par We have already characterized maximal diffusion components, so we will need to examine how the monotonicity of the optimal total cost of transporting diffusions with mass one on maximal diffusion components factor in. 
\par To wit, since the singularity is an interval pinching, maximal diffusion components include the pre-singular time components $\mathbf{t}^{-1}\left(0, T-\epsilon \right)$ and post-singular-time components $\mathbf{t}^{-1}\left(0, T+\epsilon \right)$ for any $\epsilon>0$. In the pre-singular-time maximal diffusion components, diffusions starting from a bounded measurable initial data will become smooth instantly. In post singular time components, the diffusions starting from a bounded measurable initial data supposrted on $M_1 \sqcup M_2$, will evolve separately on $M_1$ and $M_2$, whose restrictions to $M_1$ and $M_2$ become smooth instantly. In both cases mass one is preserved as well as being rotationally symmetric. In what follows we are considering the backward time parameter $\tau = T' - t$ and restrict ourselves to post-singular-time diffusion components. We will suppress the time parameter $\tau$ when the time-dependence is clear from the context. 
\par Following~\cite{AK1, AK2}, remove the poles of $M_i$ ($i = 1,2$) and identify the resulting doubly punctured sphere by $(-1,1) \times \sphere^n$ with $g_i = \phi_i^2(x) dx^2 + \psi_i^2(x) g_{\text{can}}$, where $\phi_i$ and $\psi_i$ are time-dependent smooth functions on $(-1,1)$. Choose the coordinate $r_i(x)$ to denote the distance to the north pole of $M_i$ and let ${r_i}_{~\text{max}}$ denote the distance from the north pole to the (singular) south pole, i.e., $r_{\text{max}} = \diam(M_i)$. Under this coordinate change we can write the metric $g_i$ as
\[
g_i = dr^2 + \psi_i(r)^2 g_{\text{can}}, \quad i=1,2
\]
where $\psi_i$ is $C^1$ on $(0,{r_i}_{\text{max}})$. Note that $\psi_1(0) = \psi(r_{\text{max}}) = 0$ and $\psi_1'(0) = 1$. 
\par Let $\mu(\tau)$ be a diffusion in a post-singular-time diffusion component. Then $\mu(\tau)$ is supported in $M_1\sqcup M_2$, denote the restrictions to $M_i$ by $\mu_i(\tau) d\vol_{g_i(\tau)}$ ($i = 1,2$). That is, $d\mu_i(\tau) = u_i(\,\cdot\,, \tau) \dvol_{g_i(\tau)}$ ($i = 1,2$) and $u_i$ satisfies the dynamic conjugate heat equation
\begin{equation}\label{EQ:CHE}
\frac{\del u_i}{\del \tau} = \Delta_{g_1(\tau)}u_i - R_{g_1(\tau)}u_i, \quad i = 1,2.
\end{equation}
\par In what follows we will drop the index $1$  from $u$ and $\mu$ when we are working only on $M_1$. 
\par Let $F(x,\tau)$ denote the cumulative distribution function of $\mu(\tau)$ on $M_1$ computed from the north pole of $M_1$ (the non-singular pole). $F(x,\tau)$ is defined on $[-1,1] \times (0, T_1)$ and given by
\[
F(x,\tau) = \int_{\sphere^n} \int_{-1}^{x} u(\,\cdot\,,\tau) \dvol_{g_1(\tau)}.
\]
Note that for any $\tau_0  \in (0, T_1)$, $F$ is a non-decreasing function on $[-1,1]$ with $F(-1,\tau_0) = 0$ and $F(1,\tau_0) = 1$. For the rest of these notes the smooth diffusions are assumed to be \emph{rotationally symmetric}. 
\begin{lemma}
	\label{LEM:u_r}
	Suppose rotationally symmetric $u$ and functions $ \psi_1$ and $F$ are as above. Then
	\[
	u_r = \frac{1}{\omega_n ({\psi_1})^n} \biggl( F_{rr} - n \frac{({\psi_1})_r}{{\psi_1}} F_r   \bigg) 
	\]
	where $\omega_n$ represents the volume of the unit sphere with the canonical round metric.
\end{lemma}
\begin{proof}
	We compute
	\begin{align*}
		F_r(x,\tau) &= \frac{\partial }{\partial r}\int_{\sphere^n} \int_{-1}^x u(y, \tau) \dvol_{g_1(\tau)}  \\
		&= \frac{1}{{\phi_1}(x,\tau)}\frac{\partial }{\partial x} \int_{-1}^x u(y,\tau) \vol(\sphere^n)  {\phi_1}(y,\tau) ({\psi_1}(y, \tau))^n dy \\ 
		&= \omega_n ({\psi_1}(x,\tau))^n u(x,\tau). 
	\end{align*}
	Then
	\begin{equation}
		\label{EQN:uandFr}
		u = \frac{F_r}{\omega_n ({\psi_1})^n}
	\end{equation}
	and
	\[
	u_r = \frac{\left(\omega_n ({\psi_1})^n \right)F_{rr} - n \omega_n ({\psi_1})^{n-1} ({\psi_1})_r F_r}{\omega_n^2 ({\psi_1})^{2n}} = \frac{1}{\omega_n ({\psi_1})^n} \left( F_{rr} - n \frac{({\psi_1})_r}{{\psi_1}} F_r   \right).\qedhere
	\]
\end{proof}
\begin{figure}
	\centering
	\includegraphics[width=3.2in, height=4.5in]{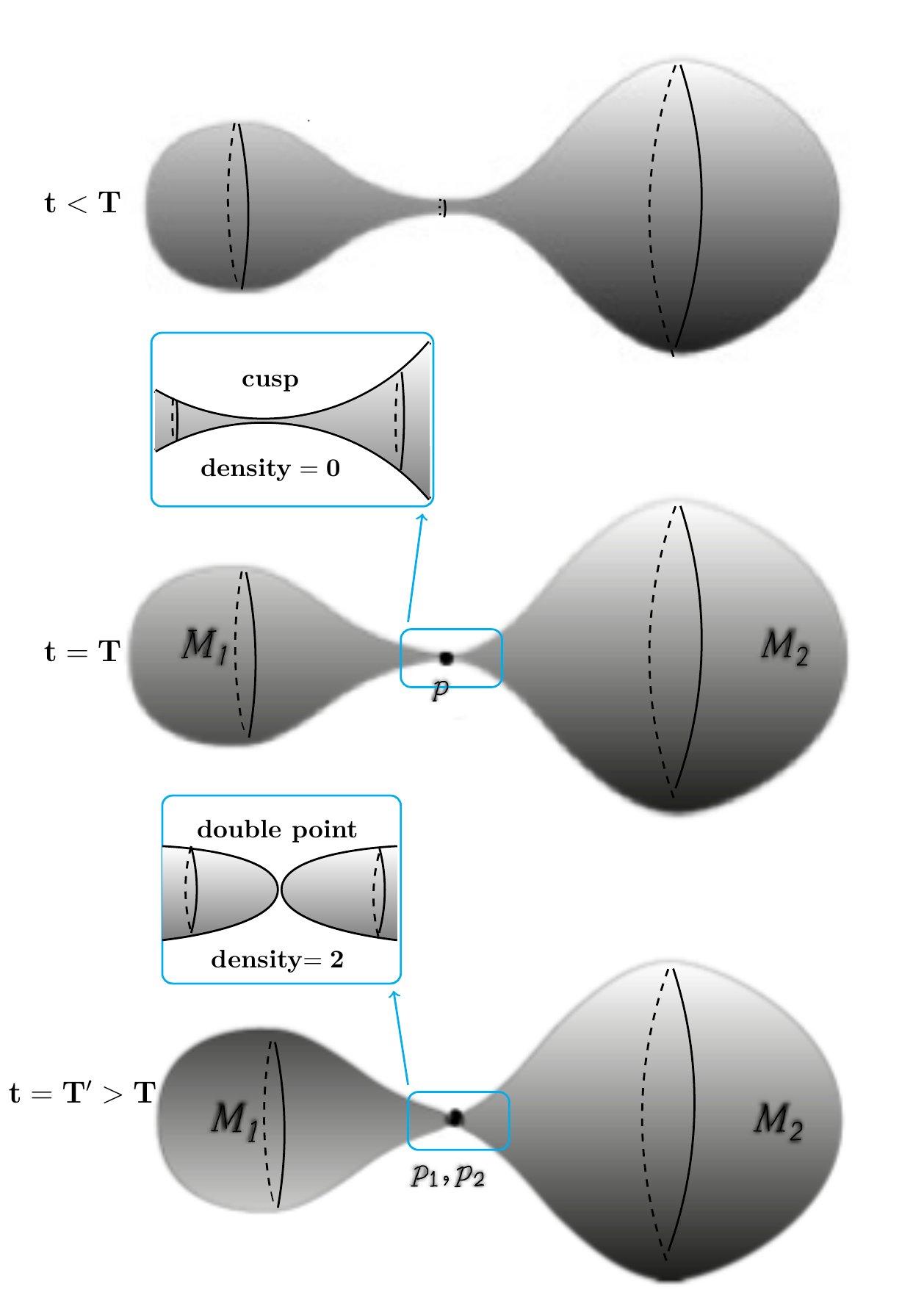}\\
	\caption{Flow through one point pinching continued as a double point is a weak super Ricci flow associated to any convex cost but is not a weak refined super Ricci flow.}
	\label{fig:onepointpinching}
\end{figure}
Next, we consider the $L^1$-norm of the cumulative distribution function $F(x, \tau)$ associated to the diffusion for $\mu(\tau)$. This will later be used in estimating the total optimal cost of transport between diffusions. Written in polar coordinates as $F(r, \tau)$,
\[
\mathcal{F_{\mu}}(\tau) := \int_0^{\diam(M_1)} F(r,\tau) dr.
\]
The main result of this section is to show that there exists a smooth diffusion measure $\mu(\tau)$ on $M_1$ such that the rate of change of $\mathcal{F}_{\mu}$ with respect to $\tau$ can be made arbitrarily large. This fact is indicative of the infinite propagation speed at time zero of parabolic evolution equations. 
\begin{proposition}
	\label{PROP:CDFBound}
	For any real number $N> 0$, there exists a diffusion $\mu_N(\tau)$ on $(M_1, g_1(\tau))$ for $\tau \in (0, T_1)$ such that 
	\[
	\dfrac{d}{d\tau} \mathcal{F}_{\mu_N} > N
	\] 
	for all sufficiently small $\tau>0$. 
\end{proposition}
\begin{proof}[\bf Proof of Proposition~\ref{PROP:CDFBound}]
	Let  $\mu(\tau)$ be any diffusion on $M_1$ with cumulative distribution function $F(x,\tau)$.  By the divergence theorem and evolution of volume form under Ricci flow, we get
	\begin{align}\label{dFdtau}
		\frac{\partial}{\partial \tau} F(x, \tau) 	& = \frac{\partial}{\partial \tau}\int_{\sphere^n} \int_{-1}^{x} u(\,\cdot\,,\tau) \dvol_{g_1(\tau)}. \notag \\
		&= \int_{{ \sphere^n}}  \int_{-1}^x \; \left(\frac{\partial u}{\partial \tau}+ R_{g_1(\tau)} u\right) \dvol_{g_1(\tau)} \notag \\ 
		&= \int_{{ \sphere^n}}  \int_{-1}^x \; \Delta_{g_1(\tau)} u ~\dvol_{g_1(\tau)} \\ 
		&= \int_{\{x\} \times { \sphere^n}} \; \left< \nabla u , {\bf n} _x \right> \dsigma \notag 
	\end{align}
	where ${\bf n}_x$ denotes the northward pointing unit normal to the hypersurface $\left\{x \right\} \times \sphere^n$ and $\mathrm{d}\sigma$ its area form.  Since the metrics $g_1(\tau)$ and the solution $u(x,\tau)$ are rotationally symmetric, we have $\n_x = \frac{\partial}{\partial r}=  \frac{1}{\phi_1} \frac{\partial}{\partial x}$ and $\nabla u = u_r \frac{\partial}{\partial r} = \frac{u_x}{(\phi_1)^2} \frac{\partial}{\partial x}$. Hence, 
	\[
	\frac{\partial}{\partial \tau} F(x, \tau) =\int_{\left\{ x \right\} \times \sphere^n} \; \frac{u_x}{(\phi_1)^3}\left| \frac{\partial}{\partial x} \right|^2 \dsigma = \int_{\left\{ x \right\} \times \sphere^n} \; \frac{u_x}{\phi_1} \dsigma = 	\frac{\partial}{\partial \tau} F(r, \tau) =\int_{\left\{ r \right\} \times { \sphere^n}}  u_r  \dsigma.
	\]
Setting $r_{\text{max}} = r(1) = \int_{-1}^1 \phi_1(y) dy =  \diam(M_1)$, and using the evolution equation for $\phi_1$ given by (\ref{phipsi evolve}), we have
	\begin{align*}
		\frac{d}{d\tau} \mathcal{F}_{\mu} = \frac{\partial}{\partial \tau} \int_{0}^{r_{\text{max}}}  F(r, \tau) dr 
		&= \frac{\partial}{\partial \tau} \int_{-1}^1  F(x,\tau)\phi_1(x, \tau) dx   \\ 
		&=  \int_{-1}^1    \left( \frac{\partial  F}{\partial \tau} \phi_1 +  F\frac{\partial \phi_1}{\partial \tau} \right)  dx \\ 
		&=  \int_{-1}^1    \left( \frac{\partial  F}{\partial \tau}+  \dfrac{F}{\phi_1}\frac{\partial \phi_1}{\partial \tau} \right)\phi_1 dx \\ 
		&=  \int_0^{r_{\text{max}}}   \left( \frac{\partial  F}{\partial \tau}+  \dfrac{F}{\phi_1}\frac{\partial \phi_1}{\partial \tau} \right)  dr \\
		&=  \int_0^{r_{\text{max}}}  \left( \frac{\partial  F}{\partial \tau}-  nF\dfrac{(\psi_1)_{rr}}{\psi_1} \right)  dr.
	\end{align*}
	To evaluate the ${\int_0^{r_{\text{max}}}  \frac{\partial  F}{\partial \tau} dr}$ term above we use (\ref{dFdtau}) and Lemma ~\ref{LEM:u_r}. Namely, 
	\begin{align*}
		\int_0^{r_{\text{max}}}  \;  \frac{\partial  F}{\partial \tau} dr
		& =  \int_0^{r_{\text{max}}}  \;  \left(\int_{\left\{ r \right\} \times {\sphere^n}}  u_r  \dsigma  \right) dr \\
		& =  \int_0^{r_{\text{max}}} \left( \int_{\sphere^n} \; u_r \psi_1^n \dvol_{\sphere^n}\right)  dr \\
		& =  \int_0^{r_{\text{max}}} \left( \vol(\sphere^n) u_r \psi_1^n \right)  dr \\
		& =  \int_0^{r_{\text{max}}} \left( F_{rr} - n\dfrac{(\psi_1)_{r}}{\psi_1}F_r \right)  dr. 
	\end{align*}
	In addition, by (\ref{Ricci}) it follows that 
	\[
	\ds{\Ric(\del r,\del r) = -n\dfrac{(\psi_1)_{rr}}{\psi}}.
	\] 
	Therefore, for any diffusion $\mu(\tau)$ on $M_1$
	\begin{equation}
	\label{EQN:ToMaximize}
	\dfrac{d}{d\tau} \mathcal{F}_{\mu} (\tau) = \frac{\partial}{\partial \tau} \int_{0}^{r_{\text{max}}}  F(r, \tau) dr = \int_0^{r_{\text{max}}} \biggl( F_{rr} - n\dfrac{(\psi_1)_{r}}{\psi_1}F_r +  F\Ric\left(\frac{\del }{\del r},\frac{\del }{\del r} \right) \bigg) dr.
	\end{equation}
	for all $\tau \in (0, T_1)$. This identity relates the time derivative to $\mathcal{F}_{\mu} (\tau)$ to spatial derivatives of $F$ and $\psi_1$. 
\par To complete the proof, we want to construct a diffusion for which the right hand side of \eqref{EQN:ToMaximize} is as large as we desire. The strategy is standard; first, we construct an initial data which makes the right hand side of \eqref{EQN:ToMaximize}  as large as we want; then, we flow this initial data under the dynamic conjugate heat flow for a short time. This way we obtain a diffusion and since the solution is given by convolution by a smooth conjugate heat kernel, the right hand side stay close to what it is at the initial time (by the standard facts from analysis, for example the derivative of convolution with the kernel is the equal to convolution with the derivative of the kernel, etc.)
\par Consider a measure that is concentrated on a small neighborhood near the singular pole (south pole is regular point for $M_1$ however in the spacetime it is where $M_1$ is connected to the neckpinch (interval) hence, singular in this sense). This measure, in a sense, approximates a Dirac mass concentrated nearby the initial singularity.  
\par Set  $\tau_0 = 0$. Recall, that $(M_1, g_1(0))$ (which in forward time is just $(M_1, g_1\left(T' \right))$) is a smooth Riemannian manifold and, by definition, 
	\[
	\psi_1(r_{\text{max}}, 0) = 0 \quad \text{and} \quad
	\lim_{r \to {r_{\text{max}}}^-}(\psi_1)_r = -1.
	\]
	Thus, since $\psi_1$ is $C^1$ on $[0, r_{\text{max}}]$, for any $\epsilon >0$ there exists $\delta_{\epsilon} >0$ such that 
	\[
	0 \leq \psi_1 \leq 2\epsilon \quad \text{ on } [r_{\text{max}} - 2\delta_{\epsilon}, r_{\text{max}}].
	\]
	Furthermore, since $\psi_1$ is smooth on $(0, r_{\text{max}})$, for any $\lambda >0$ there exists $\delta_{\lambda} >0$ such that 
	\[
	-1 - \lambda \leq (\psi_1)_r \leq -1 + \lambda \quad \text{for } r \in (r_{\text{max}} -2\delta_{\lambda}, r_{\text{max}}).
	\]
	Choose $\lambda$ so small that $\psi_1$ is monotone, strictly decreasing on $[r_{\text{max}} -2\delta_{\lambda}, r_{\text{max}}]$ and maps onto $[0, 2\epsilon]$. Take $\delta := \delta({\epsilon, \lambda})= \min(\delta_{\epsilon}, \delta_{\lambda})$ so that after making $\epsilon$ smaller if necessary, $\psi_1(r_{\text{max}} - 2\delta, 0) = 2\epsilon$ holds.
\par Let $\gamma$ be a non-negative cut-off function with $\supp(\gamma) \subset [0, 2\epsilon]$. Set 
	\[
	\gamma(\rho) = \begin{cases}
	1 &\quad \text{if } \rho \in [0, \epsilon],\\
	0 & \quad \text{if } \rho \geq 2\epsilon.
	\end{cases}
	\]
	Note that the image of $\psi_1\big|_{[r_{\text{max}} -2\delta_{\lambda}, r_{\text{max}}]}$ is contained entirely in the domain of $\gamma$ and $\psi_1(r_{\text{max}} - 2\delta, 0) = 2\epsilon$ and $\psi_1(r_{\text{max}}, 0) = 0$. Define
	\begin{align*}
		\beta_{\delta}(r) &:= \dfrac{\gamma(\psi_1) (\psi_1)_r}{\int_{r_{\text{max}} - 2\delta}^{r_{\text{max}}} \gamma(\psi_1) (\psi_1)^n (\psi_1)_r dr}\\
		&= \dfrac{\gamma(\psi_1) (\psi_1)_r}{\int_0^{2\epsilon} \gamma(\rho) \rho^n d\rho}, \quad \text{ letting } \rho = \psi_1(r).
	\end{align*}
	Note that $\beta_{\delta}$ is a smooth function with  $\supp(\beta_{\delta}) \subseteq [r_{\text{max}} - 2\delta, r_{\text{max}}]$. We  now define $F^{\delta}$ which we will take to be the cumulative distribution function for $\mu^{\delta}(0)$ ($\mu^{\delta}(0)$ will be determined later). For $r \in [0, r_{\text{max}}]$ define
	\[
	F^{\delta}_r(r,0) = \begin{cases}
	\beta_{\delta}(r) (\psi_1(r,0))^n & \text{ if } r \in [r_{\text{max}} -  2\delta, r_{\text{max}}],\\
	0			& \text{ otherwise.} 
	\end{cases}
	\]
	It remains to verify that $F^{\delta}(r, 0)$ is indeed a feasible cumulative distribution function. Note that for $r \in [r_{\text{max}} - 2\delta, r_{\text{max}}]$, 
	\[
	F^{\delta}(r, 0)  = \int_0^r F^{\delta}_{s} ds = \int_0^r \beta_{\delta}(s) (\psi_1(r', 0))^n ds = \dfrac{\int_0^r \gamma(\psi_1)(\psi_1)_{s} (\psi_1)^n ds}{\int_0^{2\epsilon} \gamma(\rho)\rho^n d\rho}.
	\]
	Thus,
	\[
	F^{\delta}(r, 0) =
	\begin{cases}
	\dfrac{\int_0^r \gamma(\phi_1)(\psi_1)_{s} (\psi_1)^n ds}{\int_0^{2\epsilon} \gamma(\rho)\rho^n d\rho} & \text{ if } r \in [r_{\text{max}} - 2\delta, r_{\text{max}}],\\
	0 																			& \text{ if } r \in [0, r_{\text{max}} -2\delta).
	\end{cases}
	\]
	Therefore, by construction $F^{\delta}(r,0)$ is non-decreasing function on $r$ and $F(0, 0) = 0$ and $F(r_{\text{max}}, 0) = 1$. We take $\mu^{\delta}(0)$ to be the unique rotationally symmetric and smooth probability measure with cumulative distribution function $F^{\delta}(r, 0)$ i.e. 
	\[
	 \mu^{\delta}(0) \left( r^{-1}\left(  \left[ r_1, r_2 \right]  \right)   \right) = F^{\delta}\left( r_2, 0 \right) - F^{\delta}\left( r_1, 0 \right).
	\]
\par Now wish to estimate the right hand side of \eqref{EQN:ToMaximize} for the probability measure $\mu^{\delta}(0)$. Starting with the first term inside the integral, since $\supp(\beta_{\delta}) \subseteq [r_{\text{max}} -2\delta, r_{\text{max}}]$, we have
	\begin{align*}
		\int_0^{r_{\text{max}}} F^{\delta}_{rr} dr & = F^{\delta}_r(r_{\text{max}}, 0) - F^{\delta}_r(0, 0) \\
		& = \beta_{\delta}(r_{\text{max}}) (\psi_1(r_{\text{max}}, 0))^n -  \beta_{\delta}(0) (\psi_1(0, 0))^n = 0.
	\end{align*}
	Next, for the third term of \eqref{EQN:ToMaximize}, since $|F^{\delta}| \leq 1$ by definition and since $(M_1, g_1(0))$ is compact and smooth,
	\[
	\bigg|\int_0^{r_{\text{max}}} F^{\delta} \Ric(\del r, \del r) dr \bigg| = \bigg|\int_{r_{\text{max}}-2\delta}^{r_{\text{max}}} F^{\delta} \Ric(\del r, \del r) dr \bigg| \leq 2\delta \sup_{r \in [r_{\text{max}} -2\delta, r_{\text{max}}]} \| \Ric\| < \infty.
	\]
	Lastly, it remains to evaluate the second term of \eqref{EQN:ToMaximize}. From the definition of $F^{\delta}_r$, 
	it follows that
	\allowdisplaybreaks[1]
	\begin{align*}
		\int_0^{r_{\text{max}}} -{n}\dfrac{(\psi_1)_r}{\psi_1}F^{\delta}_r dr &= \int_0^{r_{\text{max}}} -n(\psi_1)^{n-1}(\psi_1)_r\beta_{\delta} dr\\
		& = \int_0^{r_{\text{max}}} -n(\psi_1)^{n-1}(\psi_1)_r\dfrac{\gamma(\psi_1) (\psi_1)_r}{\ds{\int_0^{2\epsilon} \gamma(\rho) \rho^n d\rho}} dr\\
		& = \dfrac{\ds{\int_0^{r_{\text{max}}}  -n(\psi_1)^{n-1}((\psi_1)_r)^2 \gamma(\psi_1) dr}}{\ds{\int_0^{2\epsilon} \gamma(\rho) \rho^n d\rho}}\\
		& \geq(1-\lambda) \dfrac{\ds{\int_0^{r_{\text{max}}}  n(\psi_1)^{n-1}(\psi_1)_r \gamma(\psi_1) dr}}{\ds{\int_0^{2\epsilon} \gamma(\rho) \rho^n d\rho}}, \text{ since } -{(\psi_1)_r} \geq 1-\lambda\\
		& = (1-\lambda) \dfrac{\ds{\int_0^{2\epsilon} n\gamma(\rho) \rho^{n-1} d\rho}}{\ds{\int_0^{2\epsilon} \gamma(\rho) \rho^n d\rho}}\\
		& \geq (1-\lambda) \dfrac{\ds{\int_0^{\epsilon} n\gamma(\rho) \rho^{n-1} d\rho}}{\ds{\int_0^{2\epsilon} \gamma(\rho) \rho^n d\rho}}
		\\
		& \geq (1-\lambda) \dfrac{\ds{\int_0^{\epsilon} n \rho^{n-1} d\rho}}{\ds{\int_0^{2\epsilon}  \rho^n d\rho}}, \text{ by definition of } \gamma\\
		& = (1-\lambda) \dfrac{\epsilon^n}{\dfrac{(2\epsilon)^{n+1}}{n+1}}\\
		& = (1-\lambda)\dfrac{n+1}{2^{n+1}\epsilon}.
	\end{align*}
	Combining all of this we get
\[
 II(0) \geq (1-\lambda)\dfrac{n+1}{2^{n+1}\epsilon} - 2\delta \sup_{r \in [r_{\text{max}} -2\delta, r_{\text{max}}]} \| \Ric\|,
\]
where $II(0)$ is the right hand side of \eqref{EQN:ToMaximize} at time $\tau_0 = 0$ and for the probability measure $\mu^\delta(0)$.
Given any $N>0$ choose $\lambda_N >0$ and $\epsilon_N>0$ so small that 
\[
(1-\lambda_N)\dfrac{n+1}{2^{n+1}\epsilon_N} > N+1.
\]
By making $\lambda_N$ and $\epsilon_N$ smaller if necessary, we can assume
\[
\delta({\lambda_N, \epsilon_N}) \le \frac{1}{2} \left( \sup_r \|\Ric \| \right)^{-1}.
\]
Notice that $ \sup_r \|\Ric \|$ is always positive in our setting. So in conclusion
\[
 \int_0^{r_{\text{max}}} \biggl( F_{rr} - n\dfrac{(\psi_1)_{r}}{\psi_1}F_r +  F\Ric\left(\frac{\del }{\del r},\frac{\del }{\del r} \right) \bigg) dr > N.
\]
Set $\mu_N(0) := \mu^\delta(0)$. Let $\mu_N(\tau)$ be the solution to conjugate heat equation with initial data $\mu_N(0)$. Since the solution is obtained by convolution with a smooth conjugate heat kernel (see \cite{Guenther02}), it follows that $II(\tau)$ is continuous in $\tau$ so for $\tau$ sufficiently small, we have
\[
	\dfrac{d}{d\tau} \mathcal{F}_{\mu^{\delta}} (\tau) = II(\tau) > N.
\]
\end{proof}
\begin{remark}
	Alternatively, for the proof of Proposition~\ref{PROP:CDFBound}, we could have flown the delta measure at the south pole under the conjugate heat flow to obtain a smooth diffusion (the conjugate heat kernel based at the south pole) and then show largeness of $\dfrac{d}{d\tau} \mathcal{F}$ for the resulting solution by using (integrating) Li-Yau type differential Harnack estimates (there exist variety of them and all work for this purpose) for solutions to conjugate heat equation for $\tau$ small. 
\end{remark}
\subsection{Proof of Theorem ~\ref{THM:MAIN}}
\label{SEC:PROOF}
Consider $(X,d_X(\tau), \mathcal{H}_{d_X(\tau)})$ for  $\tau \in (0, T')$. The interval joining $M_1$ and $M_2$ exists for $\tau \in (0, T'-T]$ and has length $L(\tau)>0$. We will show that for any $\tau_0 \in (0, T'-T)$ 
\[
\frac{d^+ L}{d \tau}(\tau_0) = -\infty,
\]
by which we mean the upper Dini derivative of $L(\tau)$ has to be less than any negative number. To prove this we take appropriately chosen weak diffusions on $X$ supported on $M_1 \sqcup M_2$. From the definition of weak super Ricci flow, there exists a diffusion atlas comprised of maximal diffusion components. Given any two weak diffusions $\mu_1(\tau)$ and $\mu_2(\tau)$ in each of these diffusion components, 
\[
\tau \mapsto \mathcal{T}_{c_\tau}\left( \mu_1(\tau), \mu_2(\tau) \right)
\]
is non-increasing in $\tau$. 
\begin{theorem}
	\label{THM:dLdtauInfinite}
	Let $g_0 \in \mathcal{ADM}$ be an $\SO(n+1)$-invariant metric on $\sphere^{n+1}$ and $g(t)$ a solution to the Ricci flow starting at $g_0$ and which develops a interval neckpinch sinularity at $T< \infty$. Assume the resulting rotationally symmetric spherical Ricci flow spacetime is a (refined) weak super Ricci flow associated to a cost function $c_\tau = h \circ d_\tau$ for $h$ a strictly convex function or the identity. Consider the family of metric measure spaces $(X, d_X(t), \mathcal{H}_{d_X})$ defined in~\textsection\ref{SEC:Framework} consisting of a pair of disjoint manifolds $M_1$ and $M_2$ each evolving by the Ricci flow and joined by an interval of length $L(t) >0$ for $T\le t \le T'$  where $T'> T + \epsilon$ for some $\epsilon >0$. Set $\tau := T' - t$. For any $\tau_0 \in (0, \epsilon)$,
	\[
	\frac{d^+ L}{d \tau}(\tau_0) = -\infty.
	\]
\end{theorem}
\begin{proof} 
	We first prove this in the case where the function $h$ is the identity function, namely for the distance cost function. Pick an arbitrary $N>0$. From Proposition ~\ref{PROP:CDFBound}, for any $\tau_0 \in (0, T'-T)$, we can find two smooth diffusions $\mu^N_1(\tau)$ and $\mu^N_2(\tau)$ on $X$ with $\supp(\mu^N_1(\tau)) \subset M_i$, both defined on an open interval containing $\tau_0$ and such that
		\[
		\dfrac{d}{d\tau} \mathcal{F}_{\mu_i^N} > N \quad \text{ for } i = 1,2
		\]
on a perhaps smaller open interval containing $\tau_0$. Note that by Lemma~\ref{LEM:keylemma},  $\supp(\mu^N_i(\tau)) \subset M_i$ and thus $\mu^N_1(\tau)$ and $\mu^N_2(\tau)$ have disjoint supports on $X$. 
\par Denote by $F(r,\tau)$, the cumulative distribution function of $\left( \operatorname{pr} \right)_\# \mu^N_1$ computed from the north pole of $M_1$ (the non-singular pole) and let $\bar{G}(r,\tau)$ denote the cumulative distribution functions of $\left( \operatorname{pr}\right)_\# \mu^N_2$ computed from the south pole of $M_2$ (the singular pole of $M_2$). Here, $\operatorname{pr}$ is projection onto the interval $\left[ 0 , \diam M_i  \right]$ (Indeed, it is the pushout of the projection map by quotient map; see~\textsection\ref{SEC:OTRot} for details). This interval is of course time-dependent yet bounded. The probability measures $\left( \operatorname{pr}_r \right)_\# \mu^N_i$ are on the real line. Indeed, it should be clear that we are describing the 1D cumulative distribution functions to have a configuration matching Figure~\ref{fig:distributions1} in order to use the 1D total optimal cost estimate. So now we have two time-dependent probability measures on the real line with uniformly bounded support so we can apply the estimates of optimal transport in 1D. By (\ref{W1 vs L1 dist}), we have
	\[
		\mathcal{T}_{c_\tau} \left( \mu^N_1(\tau) , \mu^N_2(\tau) \right) =  \| F(r,\tau) - \bar{G}(r,\tau) \|_{L^1(\R)}.
	\]
Furthermore, if we consider $G(r,\tau)$ to be the cumulative distribution of $\left( \operatorname{pr}_r \right)_\# \mu^N_2$ computed from the north pole of $M_2$, we obviously have the relation 
\[
\bar{G} \equiv 1 - G
\]
and since $\mu^N_1(\tau)$ and $\mu^N_2(\tau)$ have disjoint supports,
	\begin{align*}
		\mathcal{T}_{c_\tau} \left(\mu^N_1(\tau) , \mu^N_2(\tau), \tau \right) &= \int_\R   \left|  F(r,\tau)  -   \bar{G}(r,\tau)  \right| dr \\ 
		&=  L(\tau) + \int_{0}^{\diam(M_1)}    F(r,\tau) dr  +  \int_{0}^{\diam(M_2)}   G(r,\tau)  dr.  
	\end{align*}
	
\begin{figure}
	\centering
	\includegraphics[width=4.2in, height=3.2in]{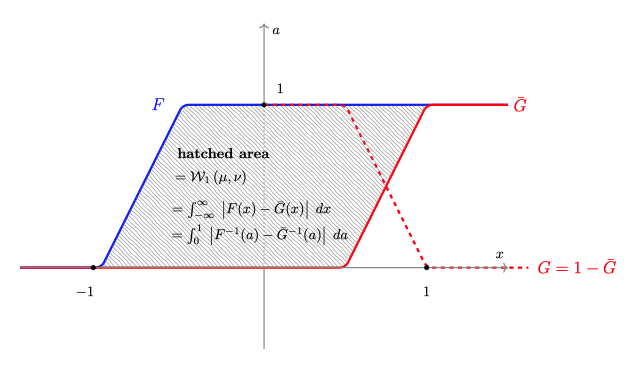}\\
	\caption{Relation between total linear cost and cumulative distributions $F$ and $\bar{G}$ in 1D.}
		\label{fig:distributions1}
\end{figure}		
Therefore,
\begin{align}
\frac{d^+}{d \tau} L  &= \frac{d^+}{d \tau}\left( \mathcal{T}_{c_\tau} \left(\mu^N_1(\tau) , \mu^N_2(\tau), \tau \right) -   \int_{0}^{\diam(M_1)}    F(r,\tau) dr -  \int_{0}^{\diam(M_2)}   G(r,\tau)  dr  \right) \notag \\ &\le \frac{d^+}{d \tau} 	\mathcal{T}_{c_\tau} \left(\mu^N_1(\tau) , \mu^N_2(\tau), \tau \right)\notag  \\ &+ \frac{d^+}{d \tau} \left( -  \int_{0}^{\diam(M_1)}    F(r,\tau) dr  \right) +  \frac{d^+}{d \tau} \left( -   \int_{0}^{\diam(M_2)}    G(r,\tau) dr  \right). \notag
\end{align}
Since the upper Dini derivatives appearing in the last line are indeed derivatives and since the total cost is non-increasing, we get
\begin{align}
\frac{d^+}{d \tau} L &\le \frac{d^+}{d \tau} 	\mathcal{T}_{c_\tau} \left(\mu^N_1(\tau) , \mu^N_2(\tau), \tau \right) - \frac{d}{d \tau} \int_{0}^{\diam(M_1)}   F(r,\tau) dr - \frac{d}{d\tau}  \int_{0}^{\diam(M_2)}  G(r,\tau)  dr  \notag \\ & \le - \frac{d}{d\tau} \mathcal{F}_{\mu^N_1} - \frac{d}{d\tau} \mathcal{F}_{\mu^N_2} \notag \\ &< -2N. \notag
\end{align}
Since $N>0$ was arbitrary, the conclusion follows. 
\par For a general (non-constant) convex cost function 
\[
c_\tau = h\left( d_\tau (x,y) \right),
\]
from (\ref{1D-totalcost}), we know the total cost equates
\[
\mathcal{T}_{c_\tau} = \int_0^1 \;\; h\left( \left|  F^{-1}_\tau (a) - \bar{G}^{-1}_\tau(a) \right| \right) \; da.
\]
Since $h$ is non-negative, convex and increasing, $h$ is almost everywhere differentiable, the one sided derivatives exist and are non-decreasing. Furthermore, 
for any $\epsilon>0$, there exists $l_\epsilon>0$ such that
\[
h'_{\pm}(x) \ge l_\epsilon, \quad \text{for $x \ge \epsilon$}.
\]
Now since the measures involved are instantly smooth and positive, the cumulative distributions are smooth, invertible and with smooth inverses; see Figure~\ref{fig:distributions2}. Thus, for any two diffusions, we can write
\begin{figure}
	\centering
	\includegraphics[width=4.2in, height=3.2in]{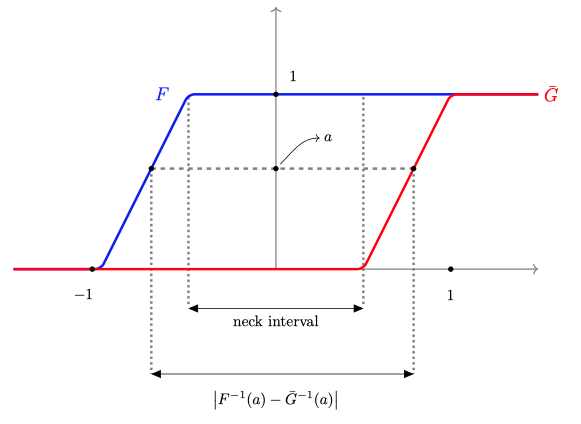}\\
	\caption{The quantity $\left|  F^{-1}(a) - \bar{G}^{-1}(a)  \right|$.}
	\label{fig:distributions2}
\end{figure}
\begin{align}
&\frac{d^\pm}{d\tau}\big|_{\tau = \tau_0} \;\;  \int_0^1 \; h\left(  \left|  F^{-1}_\tau (a) - \bar{G}^{-1}_\tau (a)  \right| \right) \; da \notag \\ &=  \int_0^1 \; h'_{\pm}\left(  \left|  F^{-1}_\tau (a) - \bar{G}^{-1}_\tau (a)  \right| \right)  \frac{d}{d\tau}\big|_{\tau = \tau_0}  \left|  F^{-1}_\tau (a) - \bar{G}^{-1}_\tau (a)  \right|  \; da \notag \\ &\ge l_\epsilon \frac{d}{d\tau}\big|_{\tau = \tau_0}  \int_{A_\epsilon} \;   \left|  F^{-1}_\tau (a) - \bar{G}^{-1}_\tau (a)  \right|  \; da \notag
\end{align}
where
\[
A^{\tau_0}_\epsilon := \left\{ a \in [0,1] ~|~   \left|  F^{-1}_\tau (a) - \bar{G}^{-1}_\tau (a)  \right|  \ge \epsilon \right\}.
\]
Therefore, we get
\[
\frac{d}{d\tau}\big|_{\tau = \tau_0}  \int_{A_\epsilon} \;   \left|  F^{-1}_\tau (a) - \bar{G}^{-1}_\tau (a)  \right|  \; da \le 0.
\]
Notice that $A^{\tau_0}_\epsilon$ is an open subset decreasing in $\epsilon$ and
\[
\lim_{\epsilon \to 0^+} A^{\tau_0}_\epsilon = [0,1]. 
\]
So we have
\begin{align}
\frac{d^+}{d \tau}\big|_{\tau = \tau_0}  \int_0^1 \;   \left|  F^{-1}_{\tau} (a) - \bar{G}^{-1}_{\tau} (a)  \right|  \; da &= \frac{d^+}{d \tau}\big|_{\tau = \tau_0} \left(  \lim_{\epsilon \to 0} \int_{A^{\tau_0}_\epsilon}\;   \left|  F^{-1}_{\tau} (a) - \bar{G}^{-1}_{\tau} (a)  \right|  \; da  \right) \notag \\ &= \limsup_{\epsilon \to 0}  \left( \frac{d^+}{d \tau}\big|_{\tau = \tau_0} \int_{A^{\tau_0}_\epsilon}\;   \left|  F^{-1}_{\tau} (a) - \bar{G}^{-1}_{\tau} (a)  \right|  \; da\right) \notag \\ &\le 0. \notag
\end{align}
which says the linear cost is nonincreasing under which we have already proven the claim.
\end{proof}
\begin{proof}[Proof of Theorem ~\ref{THM:MAIN}]
	Let $g(0) \in \mathcal{ADM}$  and $g(t)$ be a family of smooth metrics on $\sphere^{n+1}$ satisfying (\ref{EQN:RF}) for $t \in [0,T)$. By way of contradiction, assume that the neckpinch singularity that develops at $t = T$ occurs on an interval of finite length $L>0$. Adopting the framework set up in~\textsection\ref{SEC:Framework}, by Theorem ~\ref{THM:dLdtauInfinite} the quantity $\frac{\partial L}{\partial \tau}(\tau_0)$ can be made arbitrarily large and negative for any $\tau_0 \in (0, T'-T)$. That is to say, $\frac{\partial L}{\partial t}(t_0)$ can be made arbitrarily large for any $t_0 \in (T, T')$. However, this implies that the diameter of $X$ becomes unbounded as $t \to T$ contradicting \eqref{EQN:GHconvergenceBelow} and \eqref{EQN:GHconvergenceAbove}. 
To make this argument explicit, pick an arbitrary $\epsilon >0$. By  \eqref{EQN:GHconvergenceAbove} there exists  $\delta_{\epsilon}>0$ such that 
	\begin{equation}
	\label{EQN:ToContradict}
	d_{\text{GH}} \left((\sphere^{n+1}, d_{g(T)}),  (X,d_X(t))\right) < \epsilon
	\end{equation}
	for all  $t \in (T, T+\delta_{\epsilon})$. Now fix $t_0 \in (T, T+\delta_{\epsilon})$ and consider $(X,d_X(t_0))$. By assumption, $(\sphere^{n+1}, d_{g(T)})$ is compact since $L = L(0) >0$ is finite and thus $\diam(\sphere^{n+1}, d_{g(T)}) < \infty$. Therefore, from the definition of Gromov-Hausdorff distance; see, for example,~\cite[\textsection7.3]{BBI}, it follows that
	\begin{align*}
		d_{\text{GH}} \left((\sphere^{n+1}, d_{g(T)}),  (X,d_X(t_0))\right) &\geq \dfrac{1}{2} \bigl|\diam(\sphere^{n+1}, d_{g(T)})) -\diam(X,d_X(t_0))   \big|\\
		& = \dfrac{1}{2} \bigl| \diam(\sphere^{n+1}, d_{g(T)})) - \diam(M_1,g_1(t_0))  \\	
		&\phantom{= \dfrac{1}{2} \bigl|} - \diam(M_2, g_2(t_0)) - L(t_0)  \big|.
	\end{align*} 
	In addition, by Theorem ~\ref{THM:dLdtauInfinite} and taking the same $\epsilon >0$ as above, for any $N>0$ there is a sufficiently small neighborhood $U_{\epsilon,N}$ around $t_0$ such that 
	\begin{equation}
	\label{EQN:fromTHM}
	\biggl|\dfrac{L(t)-L(t_0)}{t-t_0} -N \bigg| < \epsilon
	\end{equation}
	for all $t \in U_{\epsilon, N} \setminus \{t_0\}$. Now take, for example, $t' \in U_{\epsilon, N}\cap (T, T+\delta_{\epsilon})$ with $t' > t_0$ and set
	\[
	N = \dfrac{1}{t'-t_0}\bigl(2\epsilon-\diam(\sphere^{n+1}, d_{g(T)}) + \diam(M_1, g_1(t_0)) + (\diam(M_2, g_2(t_0)) +L(t')\big) - \epsilon.
	\]
	Then, \eqref{EQN:fromTHM} implies that 
	\begin{align*}
		\epsilon < \dfrac{1}{2}\bigl(\diam(\sphere^{n+1}, d_{g(T)}) - \diam(M_1, g_1(t_0)) - \diam(M_2, g_2(t_0)) - L(t_0) \big)
	\end{align*}
	and thus
	\[
	d_{\text{GH}} \left((\sphere^{n+1}, d_{g(T)}),  (X,d_X(t_0))\right) > \epsilon
	\]
	which contradicts \eqref{EQN:ToContradict}. Therefore, if we require  that the diameter remains bounded, then the neckpinch singularity that arises at time $t = T$ must occur only at a single point; i.e., a hypersurface $\{x_0\} \times \sphere^{n}$, for some $x_0 \in (-1,1)$.\\
	The case where the neckpinch is single point pinching and under the assumption that the spacetime is weak super Ricci flow, the maximal diffusion components avoid the point neckpinches and so point neckpinches are allowed in weak super Ricci flows. 
	\par In the case where the neckpinch is single point pinching and under the assumption that the spacetime is weak refined super Ricci flow, maximal diffusion components include the single point neckpinches in their interior so an argument similar to the what we did above (based on the infinite propagation speed) will yield a contradiction again. 
\end{proof}
\subsection{Future directions}
\label{SEC:FUTURE}
\subsubsection{Asymmetric setting}
There does not yet exist a rigorous forward evolution out of a general (not necessarily rotationally symmetric) neckpinch singularity. However, the formal matched asymptotics found in~\cite[\textsection3]{AK2} predict that such evolutions should exist. The recent work on singular Ricci flows ~\cite{Kleiner-Lott} imply the existence of such smooth forward evolutions. 
\par As we have seen in the previous section, the main idea in proving the one-point pinching phenomenon using is to use the infinite propagation speed of heat type equations and the fact that heat does not travel along intervals. These facts do not require any symmetry at all. However, the other main ingredient which is reducing the computation of total cost to a 1D question uses symmetry in an essential way. So for the general case we need to get around this issue by perhaps writing the total cost as an action on geodesics which is given by a coercive Lagrangian and use an integrated Harnack inequality for diffusions instead to show the rate of change of total cost of transport between diffusions can be made arbitrarily large. 
\subsubsection{Singular Ricci flows}
The construction of spacetimes in \cite{Kleiner-Lott} is by basically by joining the spacetime slabs resulted from Ricci flow with surgery through the surviving parts. Then as the scale of surgery goes to zero, one obtains a singular Ricci flow. Since in the construction of the spacetime the singular part is thrown out one can not directly get a distance on the whole time slice at fixed post-surgery time $t_0$ however the distance survives in other ways, for example in a singular Ricci flow, there is a distance which comes from the spacetime smooth Riemannian metric and also there is Perelman's reduced $\mathcal{L}$-distance defined for the pre-surgery times however easily extendable to post-surgery times; the latter  is not really a distance function necessarily yet it shares similar properties. Then there is Topping's coupled contraction of renormalised Wasserstein distance with respect to the Perelman's reduced $\mathcal{L}$-distance under Ricci flow~\cite{Topping-L-OT} which can be used in the context of singular Ricci flows or even our Ricci flow metric measure spacetimes to show a one-point pinching result such as the one presented in these notes.

\vspace{5mm}

	\textbf{\textit{ \footnotesize “The moving finger writes, and having written moves on. Nor all thy piety nor all thy wit, can cancel half a line of it.” .... Omar Khayyam.}}


\bibliography{Neckpinch}

\providecommand{\bysame}{\leavevmode\hbox to3em{\hrulefill}\thinspace}
\providecommand{\MR}{\relax\ifhmode\unskip\space\fi MR }
\providecommand{\MRhref}[2]{%
  \href{http://www.ams.org/mathscinet-getitem?mr=#1}{#2}
}
\providecommand{\href}[2]{#2}
\begin{thebibliography}{10}

\bibitem{AKP}
L.~Ambrosio, B.~Kirchheim, and A.~Pratelli, \emph{Existence of optimal
  transport maps for crystalline norms}, Duke Math. J. \textbf{125} (2004),
  no.~2, 207--241.

\bibitem{AR-04}
L.~Ambrosio and S.~Rigot, \emph{Optimal mass transportation in the {H}eisenberg
  group}, J. Funct. Anal. \textbf{208} (2004), no.~2, 261--301.

\bibitem{AGS-GF}
Luigi Ambrosio, Nicola Gigli, and Giuseppe Savar\'{e}, \emph{Gradient flows in
  metric spaces and in the space of probability measures}, second ed., Lectures
  in Mathematics ETH Z\"{u}rich, Birkh\"{a}user Verlag, Basel, 2008.

\bibitem{AK}
Luigi Ambrosio and Bernd Kirchheim, \emph{Currents in metric spaces}, Acta
  Math. \textbf{185} (2000), no.~1, 1--80.

\bibitem{AR-14}
Luigi Ambrosio and Tapio Rajala, \emph{Slopes of {K}antorovich potentials and
  existence of optimal transport maps in metric measure spaces}, Ann. Mat. Pura
  Appl. (4) \textbf{193} (2014), no.~1, 71--87.

\bibitem{Angenent-88}
Sigurd Angenent, \emph{The zero set of a solution of a parabolic equation}, J.
  Reine Angew. Math. \textbf{390} (1988), 79--96.

\bibitem{AK1}
Sigurd Angenent and Dan Knopf, \emph{An example of neckpinching for {R}icci
  flow on {$\sphere^{n+1}$}}, Math. Res. Lett. \textbf{11} (2004), no.~4,
  493--518.

\bibitem{ACK}
Sigurd~B. Angenent, M.~Cristina Caputo, and Dan Knopf, \emph{Minimally invasive
  surgery for {R}icci flow singularities}, J. Reine Angew. Math. \textbf{672}
  (2012), 39--87.

\bibitem{AIK-degenrate}
Sigurd~B. Angenent, James Isenberg, and Dan Knopf, \emph{Degenerate neckpinches
  in {R}icci flow}, J. Reine Angew. Math. \textbf{709} (2015), 81--117.

\bibitem{AK2}
Sigurd~B. Angenent and Dan Knopf, \emph{Precise asymptotics of the {R}icci flow
  neckpinch}, Comm. Anal. Geom. \textbf{15} (2007), no.~4, 773--844.

\bibitem{ACT}
Marc Arnaudon, Kol{\'e}h{\`e}~Abdoulaye Coulibaly, and Anton Thalmaier,
  \emph{Horizontal diffusion in {$C^1$} path space}, S\'eminaire de
  {P}robabilit\'es {XLIII}, Lecture Notes in Math., vol. 2006, Springer,
  Berlin, 2011, pp.~73--94.

\bibitem{Bamler-Kleiner}
Richard Bamler and Bruce Kleiner, \emph{Uniqueness and stability of {R}icci
  flow through singularities}, arXiv:1709.04122.

\bibitem{BLM}
Lashi Bandara, Sajjad Lakzian, and Michael Munn, \emph{Geometric singularities
  and a flow tangent to the {R}icci flow}, Ann. Sc. Norm. Super. Pisa Cl. Sci.
  (5) \textbf{17} (2017), no.~2, 763--804.

\bibitem{Bao}
David Bao, \emph{On two curvature-driven problems in {R}iemann-{F}insler
  geometry}, Finsler geometry, {S}apporo 2005---in memory of {M}akoto
  {M}atsumoto, Adv. Stud. Pure Math., vol.~48, Math. Soc. Japan, Tokyo, 2007,
  pp.~19--71.

\bibitem{BGT}
Martin~T. Barlow, Alexander Grigor'yan, and Takashi Kumagai, \emph{On the
  equivalence of parabolic {H}arnack inequalities and heat kernel estimates},
  J. Math. Soc. Japan \textbf{64} (2012), no.~4, 1091--1146. \MR{2998918}

\bibitem{Bert}
J\'{e}r\^{o}me Bertrand, \emph{Existence and uniqueness of optimal maps on
  {A}lexandrov spaces}, Adv. Math. \textbf{219} (2008), no.~3, 838--851.

\bibitem{Bian-Cav}
Stefano Bianchini and Fabio Cavalletti, \emph{The {M}onge problem for distance
  cost in geodesic spaces}, Comm. Math. Phys. \textbf{318} (2013), no.~3,
  615--673.

\bibitem{BW2008}
Christoph B{\"o}hm and Burkhard Wilking, \emph{Manifolds with positive
  curvature operators are space forms}, Ann. of Math. (2) \textbf{167} (2008),
  no.~3, 1079--1097.

\bibitem{BBI}
Dmitri Burago, Yuri Burago, and Sergei Ivanov, \emph{A course in metric
  geometry}, Graduate Studies in Mathematics, vol.~33, American Mathematical
  Society, Providence, RI, 2001.

\bibitem{CFM}
Luis~A. Caffarelli, Mikhail Feldman, and Robert~J. McCann, \emph{Constructing
  optimal maps for {M}onge's transport problem as a limit of strictly convex
  costs}, J. Amer. Math. Soc. \textbf{15} (2002), no.~1, 1--26.

\bibitem{Caravenna}
Laura Caravenna, \emph{A proof of {M}onge problem in {$\Bbb R^n$} by
  stability}, Rend. Istit. Mat. Univ. Trieste \textbf{43} (2011), 31--51.

\bibitem{Cav-Hues}
Fabio Cavalletti and Martin Huesmann, \emph{Existence and uniqueness of optimal
  transport maps}, Ann. Inst. H. Poincar\'{e} Anal. Non Lin\'{e}aire
  \textbf{32} (2015), no.~6, 1367--1377.

\bibitem{Cav-Mon}
Fabio Cavalletti and Andrea Mondino, \emph{Optimal maps in essentially
  non-branching spaces}, Commun. Contemp. Math. \textbf{19} (2017), no.~6,
  1750007, 27.

\bibitem{Cd-2}
Thierry Champion and Luigi De~Pascale, \emph{The {M}onge problem for strictly
  convex norms in {$\Bbb R^d$}}, J. Eur. Math. Soc. (JEMS) \textbf{12} (2010),
  no.~6, 1355--1369.

\bibitem{CD-1}
\bysame, \emph{The {M}onge problem in {$\Bbb R^d$}}, Duke Math. J. \textbf{157}
  (2011), no.~3, 551--572.

\bibitem{Chen1999}
Chien-Hsiung Chen, \emph{Warped products of metric spaces of curvature bounded
  from above}, Trans. Amer. Math. Soc. \textbf{351} (1999), no.~12, 4727--4740.

\bibitem{CTZ}
X.~X. Chen, G.~Tian, and Z.~Zhang, \emph{On the weak {K}\"{a}hler-{R}icci
  flow}, Trans. Amer. Math. Soc. \textbf{363} (2011), no.~6, 2849--2863.

\bibitem{DeTurck}
Dennis~M. DeTurck, \emph{Deforming metrics in the direction of their {R}icci
  tensors}, J. Differential Geom. \textbf{18} (1983), no.~1, 157--162.

\bibitem{IH-property}
Simone Di~Marino, Nicola Gigli, Enrico Pasqualetto, and Elefterios Soultanis,
  \emph{Infinitesimal {H}ilbertianity of locally {$CAT(\kappa)$}-spaces},
  arXiv:1812.02086.

\bibitem{Erbar-Juillet}
Matthias Erbar and Nicolas Juillet, \emph{Smoothing and non-smoothing via a
  flow tangent to the {R}icci flow}, J. Math. Pures Appl. (9) \textbf{110}
  (2018), 123--154.

\bibitem{Evans-Gangbo}
L.~C. Evans and W.~Gangbo, \emph{Differential equations methods for the
  {M}onge-{K}antorovich mass transfer problem}, Mem. Amer. Math. Soc.
  \textbf{137} (1999), no.~653, viii+66.

\bibitem{EGZ}
Phylippe Eyssidieux, Vincent Guedj, and Ahmed Zeriahi, \emph{Convergence of
  weak {K}\"{a}hler-{R}icci flows on minimal models of positive {K}odaira
  dimension}, Comm. Math. Phys. \textbf{357} (2018), no.~3, 1179--1214.

\bibitem{Feldman-Ilmanen-Knopf}
Mikhail Feldman, Tom Ilmanen, and Dan Knopf, \emph{Rotationally symmetric
  shrinking and expanding gradient {K}\"ahler-{R}icci solitons}, J.
  Differential Geom. \textbf{65} (2003), no.~2, 169--209.

\bibitem{Feldman-McCann}
Mikhail Feldman and Robert~J. McCann, \emph{Monge's transport problem on a
  {R}iemannian manifold}, Trans. Amer. Math. Soc. \textbf{354} (2002), no.~4,
  1667--1697.

\bibitem{FOT}
Masatoshi Fukushima, Yoichi Oshima, and Masayoshi Takeda, \emph{Dirichlet forms
  and symmetric {M}arkov processes}, extended ed., De Gruyter Studies in
  Mathematics, vol.~19, Walter de Gruyter \& Co., Berlin, 2011.

\bibitem{Gigli-optimalmaps}
Nicola Gigli, \emph{Optimal maps in non branching spaces with {R}icci curvature
  bounded from below}, Geom. Funct. Anal. \textbf{22} (2012), no.~4, 990--999.

\bibitem{Gigli-Mantegazza}
Nicola Gigli and Carlo Mantegazza, \emph{A flow tangent to the {R}icci flow via
  heat kernels and mass transport}, Adv. Math. \textbf{250} (2014), 74--104.

\bibitem{GRS}
Nicola Gigli, Tapio Rajala, and Karl-Theodor Sturm, \emph{Optimal maps and
  exponentiation on finite-dimensional spaces with {R}icci curvature bounded
  from below}, J. Geom. Anal. \textbf{26} (2016), no.~4, 2914--2929.

\bibitem{GZ}
Vincent Guedj and Ahmed Zeriahi, \emph{Regularizing properties of the twisted
  {K}\"{a}hler-{R}icci flow}, J. Reine Angew. Math. \textbf{729} (2017),
  275--304.

\bibitem{Guenther02}
Christine~M. Guenther, \emph{The fundamental solution on manifolds with
  time-dependent metrics}, J. Geom. Anal. \textbf{12} (2002), no.~3, 425--436.

\bibitem{Hamilton1982}
Richard~S. Hamilton, \emph{Three-manifolds with positive {R}icci curvature}, J.
  Differential Geom. \textbf{17} (1982), no.~2, 255--306.

\bibitem{Hamilton-Four-Manifolds}
\bysame, \emph{Four-manifolds with positive isotropic curvature}, Comm. Anal.
  Geom. \textbf{5} (1997), no.~1, 1--92.

\bibitem{Haslhofer-Naber}
Robert Haslhofer and Aaron Naber, \emph{Characterizations of the {R}icci flow},
  J. Eur. Math. Soc. (JEMS) \textbf{20} (2018), no.~5, 1269--1302.

\bibitem{IKS}
James Isenberg, Dan Knopf, and Nata\v{s}a \v{S}e\v{s}um, \emph{Ricci flow
  neckpinches without rotational symmetry}, Comm. Partial Differential
  Equations \textbf{41} (2016), no.~12, 1860--1894.

\bibitem{Kleiner-Lott}
Bruce Kleiner and John Lott, \emph{Singular {R}icci flows {I}}, Acta Math.
  \textbf{219} (2017), no.~1, 65--134.

\bibitem{Kopfer}
Eva Kopfer, \emph{Super-{R}icci flows and improved gradient and transport
  estimates}, Probab. Theory Related Fields \textbf{175} (2019), no.~3-4,
  897--936.

\bibitem{Kopfer-Sturm}
Eva Kopfer and Karl-Theodor Sturm, \emph{Heat flow on time-dependent metric
  measure spaces and super-{R}icci flows}, Comm. Pure Appl. Math. \textbf{71}
  (2018), no.~12, 2500--2608.

\bibitem{Lakzian-DHE-FRF}
Sajjad Lakzian, \emph{Differential {H}arnack estimates for positive solutions
  to heat equation under {F}insler-{R}icci flow}, Pacific J. Math. \textbf{278}
  (2015), no.~2, 447--462.

\bibitem{lakzian-RF}
\bysame, \emph{Intrinsic flat continuity of {R}icci flow through neckpinch
  singularities}, Geom. Dedicata \textbf{179} (2015), 69--89.

\bibitem{Lakzian-Munn}
Sajjad Lakzian and Michael Munn, \emph{Metric perspectives of the {R}icci flow
  applied to disjoint unions}, Anal. Geom. Metr. Spaces \textbf{2} (2014),
  282--293.

\bibitem{Lierl-18}
Janna Lierl, \emph{Parabolic {H}arnack inequality on fractal-type metric
  measure {D}irichlet spaces}, Rev. Mat. Iberoam. \textbf{34} (2018), no.~2,
  687--738.

\bibitem{Lierl20}
\bysame, \emph{Parabolic {H}arnack inequality for time-dependent non-symmetric
  {D}irichlet forms}, J. Math. Pures Appl. (9) \textbf{140} (2020), 1--66.

\bibitem{Lions-Magenes}
J.-L. Lions and E.~Magenes, \emph{Non-homogeneous boundary value problems and
  applications. {V}ol. {I}}, Springer-Verlag, New York-Heidelberg, 1972,
  Translated from the French by P. Kenneth, Die Grundlehren der mathematischen
  Wissenschaften, Band 181. \MR{0350177}

\bibitem{Lott-OT}
John Lott, \emph{Optimal transport and {P}erelman's reduced volume}, Calc. Var.
  and Partial Differential Equations \textbf{36} (2009), no.~5, 49--84.

\bibitem{McCann-polar}
Robert~J. McCann, \emph{Polar factorization of maps on {R}iemannian manifolds},
  Geom. Funct. Anal. \textbf{11} (2001), no.~3, 589--608.

\bibitem{McCann-Topping}
Robert~J. McCann and Peter~M. Topping, \emph{Ricci flow, entropy and optimal
  transportation}, Amer. J. Math. \textbf{132} (2010), no.~3, 711--730.

\bibitem{Ohta-MCP}
Shin-ichi Ohta, \emph{On the measure contraction property of metric measure
  spaces}, Comment. Math. Helv. \textbf{82} (2007), no.~4, 805--828.

\bibitem{Oshima2004}
Yoichi Oshima, \emph{Time-dependent {D}irichlet forms and related stochastic
  calculus}, Infin. Dimens. Anal. Quantum Probab. Relat. Top. \textbf{7}
  (2004), no.~2, 281--316.

\bibitem{Pazy}
A.~Pazy, \emph{Semigroups of linear operators and applications to partial
  differential equations}, Applied Mathematical Sciences, vol.~44,
  Springer-Verlag, New York, 1983.

\bibitem{Perelman-entropy}
Grisha Perelman, \emph{The entropy formula for the {R}icci flow and its
  geometric applications}, arXiv:math/0211159 (2002).

\bibitem{Perelman-RFsolns}
\bysame, \emph{Finite extinction time for the solutions to the {R}icci flow on
  certain three-manifolds}, arXiv:math/0307245 (2003).

\bibitem{Perelman-RFWS}
\bysame, \emph{{R}icci flow with surgery on three-manifolds},
  arXiv:math/0303109 (2003).

\bibitem{RR}
Svetlozar~T. Rachev and Ludger R\"{u}schendorf, \emph{Mass transportation
  problems. {V}ol. {I}}, Probability and its Applications (New York),
  Springer-Verlag, New York, 1998, Theory.

\bibitem{RR-PDE}
Michael Renardy and Robert~C. Rogers, \emph{An introduction to partial
  differential equations}, second ed., Texts in Applied Mathematics, vol.~13,
  Springer-Verlag, New York, 2004. \MR{2028503}

\bibitem{Santambrogio}
Filippo Santambrogio, \emph{Optimal transport for applied mathematicians,
  calculus of variations, pdes, and modeling}, Progress in Nonlinear
  DifferentialEquations and Their Applications, Birk\"{a}user Springer, Basel,
  2015.

\bibitem{Simon-Pinch}
Miles Simon, \emph{A class of {R}iemannian manifolds that pinch when evolved by
  {R}icci flow}, Manuscripta Math. \textbf{101} (2000), no.~1, 89--114.

\bibitem{Song-Tian}
Jian Song and Gang Tian, \emph{The {K}\"{a}hler-{R}icci flow through
  singularities}, Invent. Math. \textbf{207} (2017), no.~2, 519--595.

\bibitem{Sturmian}
J.~C.~F. Sturm, \emph{M\'emoire sur les \'equations diff\'erentielles
  lin\'eaires de second ordre}, J. Math. Pures Appl. \textbf{1} (1836),
  106--186.

\bibitem{Sturm-DirIII}
K.~T. Sturm, \emph{Analysis on local {D}irichlet spaces. {III}. {T}he parabolic
  {H}arnack inequality}, J. Math. Pures Appl. (9) \textbf{75} (1996), no.~3,
  273--297. \MR{1387522}

\bibitem{Sturm-HK}
\bysame, \emph{Diffusion processes and heat kernels on metric spaces}, Ann.
  Probab. \textbf{26} (1998), no.~1, 1--55.

\bibitem{Sturm-DirI}
Karl-Theodor Sturm, \emph{Analysis on local {D}irichlet spaces. {I}.
  {R}ecurrence, conservativeness and {$L^p$}-{L}iouville properties}, J. Reine
  Angew. Math. \textbf{456} (1994), 173--196. \MR{1301456}

\bibitem{Sturm-DirII}
\bysame, \emph{Analysis on local {D}irichlet spaces. {II}. {U}pper {G}aussian
  estimates for the fundamental solutions of parabolic equations}, Osaka J.
  Math. \textbf{32} (1995), no.~2, 275--312.

\bibitem{Sturm-2006-II}
\bysame, \emph{On the geometry of metric measure spaces. {II}}, Acta Math.
  \textbf{196} (2006), no.~1, 133--177.

\bibitem{Sturm-SRF}
\bysame, \emph{Super-{R}icci flows for metric measure spaces}, J. Funct. Anal.
  \textbf{275} (2018), no.~12, 3504--3569.

\bibitem{Sudakov}
V.~N. Sudakov, \emph{Geometric problems in the theory of infinite-dimensional
  probability distributions}, Proc. Steklov Inst. Math. (1979), no.~2, i--v,
  1--178, Cover to cover translation of Trudy Mat. Inst. Steklov {{\bf{1}}41}
  (1976).

\bibitem{Topping-L-OT}
Peter Topping, \emph{L-optimal transportation for {R}icci flow}, J. Reine
  Angew. Math. \textbf{636} (2009), 93--122.

\bibitem{Topping-FoundationsOT}
\bysame, \emph{Ricci flow: The foundations via optimal transportation}, Optimal
  Transportation, Theory and Applications, LMS lecture notes series, vol. 413,
  CUP, 2014.

\bibitem{Topping-McCann}
Peter Topping and Robert McCann, \emph{Ricci flow entropy and optimal
  transport}, American Journal of Mathematics \textbf{132} (2010), 711--730.

\bibitem{Trudinger-Wang}
Neil~S. Trudinger and Xu-Jia Wang, \emph{On the {M}onge mass transfer problem},
  Calc. Var. Partial Differential Equations \textbf{13} (2001), no.~1, 19--31.

\bibitem{Villani-TOT}
C{\'e}dric Villani, \emph{Topics in optimal transportation}, Graduate Studies
  in Mathematics, vol.~58, American Mathematical Society, Providence, RI, 2003.

\bibitem{Zhou}
Zhou Zhang, \emph{K\"{a}hler-{R}icci flow with degenerate initial class},
  Trans. Amer. Math. Soc. \textbf{366} (2014), no.~7, 3389--3403.

\end{thebibliography}
\bibliographystyle{amsplain}
\end{document}